\documentclass[11pt,reqno]{amsart}
\usepackage{amsmath}
\usepackage{amsthm}
\usepackage{amsfonts}
\usepackage{amssymb}
\usepackage{tikz}
\usepackage{float}
\usepackage[all]{xy}
\usepackage{hyperref}
\usepackage{enumerate}
\usepackage{enumitem}

\newcommand{\op}{\ensuremath{^{\mathrm{op}}}}

\newcommand{\bA}{\mathbb{A}}

\newcommand{\bZ}{\mathbb{Z}}

\newcommand{\sfS}{\mathsf{S}}
\newcommand{\sfT}{\mathsf{T}}

\newcommand{\N}{\mathbb{N}}

\newcommand{\cA}{{\mathcal A}}
\newcommand{\cB}{{\mathcal B}}
\newcommand{\cC}{{\mathcal C}}
\newcommand{\cD}{{\mathcal D}}
\newcommand{\cE}{{\mathcal E}}

\newcommand{\cG}{{\mathcal G}}

\newcommand{\cI}{{\mathcal I}}

\newcommand{\cP}{{\mathcal P}}

\newcommand{\cX}{{\mathcal X}}
\newcommand{\cY}{{\mathcal Y}}

\newcommand{\ob}{\operatorname{ob}}
\newcommand{\md}{\operatorname{mod}}
\newcommand{\Md}{\operatorname{Mod}}

\newcommand{\im}{\operatorname{im}}

\newcommand{\proj}{\operatorname{proj}}

\newcommand{\res}{\operatorname{res}}

\newcommand{\emphbf}[1]{\emph{\textbf{#1}}}

\newcommand{\unit}[2]{\alpha^{#1\dashv #2}}
\newcommand{\counit}[2]{\beta^{#1\dashv #2}}
\newcommand{\adjiso}[2]{\phi^{#1\dashv #2}}

\newcommand{\relGinj}[2]{\cG\cI_{#1}(#2 )}
\newcommand{\relGproj}[2]{\cG\cP_{#1}(#2 )}

\newcommand{\Mor}{\operatorname{Mor}}
\newcommand{\mor}{\operatorname{mor}}

\newcommand{\pdim}{\operatorname{proj.dim}}

\newcommand{\flatdim}{\operatorname{flat.dim}}

\newcommand{\Hom}{\operatorname{Hom}}

\newcommand{\Ext}{\operatorname{Ext}}

\newcommand{\Tor}{\operatorname{Tor}}

\newcommand{\Ker}{\operatorname{Ker}}
\newcommand{\Coker}{\operatorname{Coker}}

\usetikzlibrary{matrix,arrows}
\title[A generalization of the Nakayama functor]{A generalization of the Nakayama functor}

\date{\today}


\author{Sondre Kvamme}

\address{Laboratoire de Math\'ematiques d'Orsay, Universit\'e Paris-Sud, 91405 Orsay, France}

\email{sondre.kvamme@u-psud.fr}

\begin{document}

\newtheorem{Theorem}[equation]{Theorem}
\newtheorem{Lemma}[equation]{Lemma}
\newtheorem{Corollary}[equation]{Corollary}
\newtheorem{Proposition}[equation]{Proposition}

\theoremstyle{definition}
\newtheorem{Definition}[equation]{Definition}
\newtheorem{Example}[equation]{Example}
\newtheorem{Remark}[equation]{Remark}
\newtheorem{Setting}[equation]{Setting}
\newtheorem{Conjecture}[equation]{Conjecture}

\thanks{This is part of the authors PhD thesis. The author thanks Gustavo Jasso, Julian K\"ulshammer, Rosanna Laking, and  Jan Schr\"oer for helpful discussions and comments and on a previous version of this paper. He would also like to thank the referee for helpful comments and suggestions. The work was made possible by the funding provided by the \emph{Bonn International Graduate School in Mathematics}.}


\begin{abstract}
In this paper we introduce a generalization of the Nakayama functor for finite-dimensional algebras. This is obtained by abstracting its interaction with the forgetful functor to vector spaces. In particular, we characterize the Nakayama functor in terms of an ambidextrous adjunction of monads and comonads. In the second part we develop a theory of Gorenstein homological algebra for such Nakayama functor. We obtain analogues of several classical results for Iwanaga-Gorenstein algebras. One of our main examples is the module category $\Lambda\text{-}\Md$ of a $k$-algebra $\Lambda$, where $k$ is a commutative ring and $\Lambda$ is finitely generated projective as a $k$-module. 
\end{abstract}

\maketitle

\setcounter{tocdepth}{2}
\numberwithin{equation}{subsection}
\tableofcontents


\section{Introduction}

\subsection{Introduction}
Let $\Lambda$ be a finite-dimensional $k$-algebra, and let $D(\Lambda)=\Hom_k(\Lambda,k)$ be the $k$-dual. Consider the Nakayama functor 
\[
\nu= D(\Lambda)\otimes_{\Lambda}-\colon \Lambda\text{-}\md\to \Lambda\text{-}\md
\]
It is used to define the Auslander-Reiten translation, which is a fundamental tool for doing Auslander-Reiten theory. It therefore plays an important role in determining the structure of  $\Lambda\text{-}\md$. On the other hand, we also have the well developed theory of Gorenstein homological algebra, see for example \cite{AB69,AB89,Buc86,EEG08,EJ95,EJ11,EJ11a,Hol04,J07}. This theory has been studied in the context of finite-dimensional algebras, see for example \cite{AR91a,BK08,Che10,She16}, and in this case many of the statements are linked to the Nakayama functor. For example, a finite-dimensional algebra $\Lambda$ is Iwanaga-Gorenstein if and only if 
\[
\pdim D(\Lambda)_{\Lambda}<\infty \quad \text{and} \quad \pdim {}_{\Lambda}D(\Lambda)<\infty.
\]
and in this case it follows from a result by Zaks \cite{Zak69} that  
\[
\pdim D(\Lambda)_{\Lambda}= \pdim {}_{\Lambda}D(\Lambda).
\]
where this number is the Gorenstein dimension of $\Lambda$. These are statements about the vanishing of the left and right derived functor of the Nakayama functor and its right adjoint, respectively. 

The aim of this paper is twofold. We first give a generalization of the Nakayama functor. This is obtained by observing that it interacts nicely with the adjoint pair $f_!\dashv f^*$, where 
\[
f^*\colon \res^{\Lambda}_k\colon \Lambda\text{-}\md \to k\text{-}\md
\]
is the restriction functor, and $f_!:=\Lambda\otimes_k -\colon k\text{-}\md\to \Lambda\text{-}\md$ is its left adjoint. More precisely, we make the following definition.

\begin{Definition}[Definition \ref{Nakayama functor for adjoint pair}]\label{Nakayama functor for adjoint pair introduction}
Let $\cD$ be a preadditive category and $\cA$ an abelian category, and let $f^*\colon \cA\to \cD$ be a faithful functor with left adjoint $f_!\colon \cD\to \cA$. A \emphbf{Nakayama functor} relative to $f_!\dashv f^*$ is a functor $\nu\colon \mathcal{A}\to \mathcal{A}$ with  a right adjoint $\nu^-$ satisfying:
\begin{enumerate}
\item $\nu\circ f_!$ is right adjoint to $f^*$;
\item The unit of $\nu\dashv\nu^-$ induces an isomorphism $f_!\xrightarrow{\cong} \nu^-\circ \nu \circ f_!$ when precomposed with $f_!$.
\end{enumerate}
\end{Definition}
The Nakayama functor is unique up to unique equivalence, see Theorem \ref{Uniqueness Nakayama functor adjoint pairs}. Furthermore, a choice of a Nakayama functor is the same as a choice of an isomorphism of certain categories, or a choice of an isomorphism of certain monads, see Proposition \ref{Extending Nakayama functor Adjoint pair} and Theorem \ref{Isomorphism of monads adjoint pairs}.

We have the following example.

\begin{Example}[Example \ref{One algebra}]\label{Commutative ring}
Let $k$ be a commutative ring, and let $\Lambda$ be a $k$-algebra which is finitely generated projective as a $k$-module. Orders over complete regular local rings \cite{Iya01} are examples of such algebras. Consider the restriction functor 
\[
f^*:= \res^{\Lambda}_k\colon \Lambda\text{-}\Md \to k\text{-}\Md
\]
with left adjoint $f_!:=\Lambda\otimes_k -\colon k\text{-}\Md\to \Lambda\text{-}\Md$. Then the functor 
\[
\nu_{\Lambda\text{-}\Md}:=D(\Lambda)\otimes_{\Lambda}-\colon \Lambda\text{-}\Md\to \Lambda\text{-}\Md
\] 
is a Nakayama functor relative to $f_!\dashv f^*$, where $D:=\Hom_k(-,k)$ is the algebraic dual.
\end{Example} 
This gives the classical Nakayama functor for a finite-dimensional algebra. For more examples see the end of Subsection \ref{Nakayama functors subsection}. In fact, for any commutative ring $k$, any $k$-linear abelian category $\cB$, and any small, $k$-linear, locally bounded and Hom-finite category $\cC$, the functor category $\cA=\cB^{\cC}$ has a Nakayama functor relative to an adjoint pair, see Theorem \ref{Theorem:5}. Such categories $\cC$ are also studied in \cite{DSS17}.

The notion of a Nakayama functor relative to an adjoint pair can also be phrased in terms of monads and comonads. In fact, if $f_!\dashv f^*$ has a Nakayama functor, then $f^*$ is both faithful and exact. Hence, if $\cD$ is abelian, then by Beck's monadicity theorem (see Theorem 4.4.4 in \cite{Bor94a} or Theorem \ref{Weak Beck's monadicity} below) we have an equivalence $\cD^{\sfT}\cong \cA$ commuting up to isomorphism with $f^*$ and the forgetful functor $U^{\sfT}\colon \cD^{\sfT}\to \cD$. Here $\cD^{\sfT}$ is the Eilenberg-Moore category of the induced monad $\sfT$ on $\cD$ coming from the adjunction, see Example \ref{example:1} and Definition \ref{Eilenberg-Moore category}.  We determine necessary and sufficient conditions on $\sfT$ for the existence of a Nakayama functor.

\begin{Theorem}[Theorem \ref{Admissible adjunction theorem}]\label{Admissible adjunction theorem introduction}
Let $\sfT$ be a monad on an abelian category $\cB$, let $\cB^{\sfT}$ be the Eilenberg-Moore category of $\sfT$, and let $U^{\sfT}\colon \cB^{\sfT} \to \cB$ be the forgetful functor with left adjoint $F^{\sfT}\colon \cB\to \cB^{\sfT}$. Then the adjoint pair $F^{\sfT}\dashv U^{\sfT}$ has a Nakayama functor if and only if there exists a comonad $\sfS$ and an ambidextrous adjunction $\sfS\dashv \sfT\dashv \sfS$.   
\end{Theorem}
See Definition \ref{adjoint monads and comonads} for the notion of adjunction between monads and comonads. Note that it is not sufficient to just assume there exists an ambidextrous adjunction $S\dashv T\dashv S$ where $T$ is the underlying functor of $\sfT$, since the induced comonad structures on $S$ coming from the adjunctions $S\dashv T$ and $T\dashv S$ might be different. It follows from Beck's monadicity theorem and Theorem \ref{Admissible adjunction theorem introduction} that if $\cD$ is abelian and $f^*\colon \cA\to \cD$ is faithful and exact, then $f_!\dashv f^*$ has a Nakayama functor if and only if there exists a comonad $\sfS$ and an ambidextrous adjunction $\sfS\dashv \sfT\dashv \sfS$, where $\sfT$ is the induced monad on $\cD$ coming from $f_!\dashv f^*$.

The second aim of this paper is to develop a theory of Gorenstein homological algebra for such Nakayama functor. In fact, our constructions and results in this part only depends on the endofunctor $P:=f_!\circ f^*\colon \cA\to \cA$, and not on the adjoint pair itself. 

\begin{Definition}[Definition \ref{definition:N1}]
Let $\cA$ be an abelian category, and let $P\colon \cA\to \cA$ be a generating functor. A \emphbf{Nakayama functor} relative to $P$ is a functor $\nu\colon \cA\to \cA$ with a right adjoint $\nu^-$, and satisfying the following:
\begin{enumerate}
	\item $\nu\circ P$ is right adjoint to $P$\label{Nakayama1};
	\item The unit $\lambda\colon 1_{\cA}\to \nu^-\circ \nu$ induces an isomorphism on objects of the form $P(A)$ for $A\in \cA$. \label{Nakayama2}
\end{enumerate}
\end{Definition}
Here $P$ is generating if the the image of $P$ is a generating subcategory of $\cA$, see Definition \ref{Generating and cogenerating}. 

We define the analogue of Gorenstein projective and injective modules.

\begin{Definition}[Definition \ref{Gorenstein objects for comonad with Nakayama functor paper}]\label{definition:I1.5}
Let $P\colon \cA\to \cA$ be a generating functor with a Nakayama functor $\nu$ relative to $P$, and let $I:=\nu\circ P$ be the right adjoint of $P$.
\begin{enumerate}
\item An object $G\in \cA$ is \emphbf{Gorenstein $P$-projective} if there exists an exact sequence
\[
Q_{\bullet}=\cdots \xrightarrow{f_2} Q_{1} \xrightarrow{f_1} Q_0\xrightarrow{f_0} Q_{-1} \xrightarrow{f_{-1}} \cdots 
\]
in $\cA$, where $Q_i=P(A_i)$ for $A_i\in \cA$, such that the complex $\nu(Q_{\bullet})$ is exact, and with $Z_0(Q_{\bullet})=\Ker f_0=G$;
\item An object $G\in \cA$ is \emphbf{Gorenstein $I$-injective} if there exists an exact sequence
\[
J_{\bullet}=\cdots \xrightarrow{g_2} J_{1} \xrightarrow{g_1} J_0\xrightarrow{g_0} J_{-1} \xrightarrow{g_{-1}} \cdots 
\]
in $\cA$, where $J_i=I(A_i)$ for $A_i\in \cA$, such that the complex $\nu^-(J_{\bullet})$ is exact, and with $Z_0(J_{\bullet})=\Ker g_0=G$.
\end{enumerate} 
\end{Definition} 
For a finite-dimensional algebra $\Lambda$ with $P=f^*f_!$, a $\Lambda$-module is Gorenstein $P$-projective or $I$-injective if and only if it is Gorenstein projective or injective, respectively. Let $\relGproj{P}{\cA}$ and $\relGinj{I}{\cA}$ denote the subcategories of Gorenstein $P$-projective and $I$-injective objects. In Proposition \ref{resolving and coresolving} we show that these are resolving and coresolving subcategories, respectively. Hence, we can define the Gorenstein $P$-projective dimension and $I$-injective dimension $\dim_{\relGproj{P}{\cA}}(-)$  and $\dim_{\relGinj{I}{\cA}}(-)$ by taking resolutions of Gorenstein $P$-projective objects and coresolutions of Gorenstein $I$-injective objects, as described in Subsection \ref{Properties of subcategories}.

We define an analog of Iwanaga-Gorenstein algebras.

\begin{Definition}[Definition \ref{definition:N2}]\label{definition:I3}
Let $P\colon \cA\to \cA$ be a generating functor with a Nakayama functor $\nu$ relative to $P$.  We say that $P$ is \emphbf{Iwanaga-Gorenstein} if there exists an $m\geq 0$ such that $L_i\nu(A)=0$ and $R^i\nu^-(A)=0$ for all $A\in \cA$ and $i>m$. 
\end{Definition}

Note that the left derived functor of $\nu$ and the right derived functor of $\nu^-$ exists without assuming enough projectives or injectives, see Lemma \ref{lemma:N1}. For a finite-dimensional algebra $\Lambda$, the functor $P$ is Iwanaga-Gorenstein if and only if $\Lambda$ is Iwanaga-Gorenstein. More generally, the functor $P=f_!\circ f^*$ in Example \ref{Commutative ring} is Iwanaga-Gorenstein if and only if
\[
\pdim {}_{\Lambda}D(\Lambda)<\infty \quad \text{and} \quad \pdim D(\Lambda)_{\Lambda}<\infty.
\]

We show that if $P$ is Iwanaga-Gorenstein, then there exists a simple description of $\relGproj{P}{\cA}$ and $\relGinj{I}{\cA}$, see Theorem \ref{thm:N1}. Our main result in this part is the following, which is an analogue of Iwanaga's result on the injective dimension of Iwanaga-Gorenstein algebras, and a generalization of \cite[Theorem 2.3.3]{Che10}.

\begin{Theorem}[Theorem \ref{thm:N2}]\label{theorem:I5}
The following are equivalent:
\begin{enumerate}[label=(\alph*)]
	\item $P$ is Iwanaga-Gorenstein;
	\item $\dim_{\relGproj{P}{\cA}}(\cA)< \infty$;
	\item $\dim_{\relGinj{I}{\cA}}(\cA)< \infty$.
\end{enumerate}
Moreover, if this holds, then the following numbers coincide:
\begin{enumerate}
\item $\dim_{\relGproj{P}{\cA}}(\cA)$;
\item $\dim_{\relGinj{I}{\cA}}(\cA)$;
\item The smallest integer $s$ such that $L_i\nu(A)=0$ for all $i>s$ and $A\in \cA$;
\item The smallest integer $t$ such that $R^i\nu^-(A)=0$ for all $i>t$ and $A\in \cA$.
\end{enumerate}
We say that $P$ is $n$\emphbf{-Gorenstein} if this common number is $n$.
\end{Theorem}

For a finite-dimensional algebra $\Lambda$, the endofunctor $P$ is $n$-Gorenstein if and only if $\Lambda$ is $n$-Gorenstein. For the endofunctor $P=f_!\circ f^*$ in Example \ref{Commutative ring} we get the following corollary.

\begin{Corollary}[See Theorem \ref{thm:N3}]\label{corollary:I1}
Let $k$ be a commutative ring, and let $\Lambda$ be a $k$-algebra which is finitely generated and projective as a $k$-module. Assume that
\[
\pdim D(\Lambda)_{\Lambda}<\infty \quad \text{and} \quad \pdim {}_{\Lambda}D({\Lambda})<\infty.
\]
Then 
\[
\pdim D(\Lambda)_{\Lambda}=\pdim {}_{\Lambda}D(\Lambda).
\] 
\end{Corollary}
Note that this dimension is $n$ if and only if $P$ is $n$-Gorenstein.  We also have a version of this result for small, locally bounded, and Hom-finite categories, see Theorem \ref{thm:N3}. As far as we know, these results have not been written down before.

The paper is organized as follows. In Section \ref{Preliminaries} we recall the necessary preliminaries on adjunctions and monads that we need. In Section \ref{Nakayama functors section} we define Nakayama functors relative to adjoint pairs and endofunctors, and we provide examples. Furthermore, we give some different characterizations for when such Nakayama functors exist, see Proposition \ref{Extending Nakayama functor Adjoint pair}, Corollary \ref{Computing the Nakayama functor} and Theorem \ref{Isomorphism of monads adjoint pairs}, as well as proving a uniqueness result for such Nakayama functors, see Theorem \ref{Uniqueness Nakayama functor adjoint pairs}. We end the section by proving Theorem \ref{Admissible adjunction theorem introduction} above. In Section \ref{Gorenstein categories for comonads} we develop a theory of Gorenstein homological algebra for Nakayama functors relative to endofunctors. In particular, we define Gorenstein $P$-projective and $I$-injective objects and show that they form a resolving and coresolving subcategory, see Proposition \ref{resolving and coresolving}. We also define the notion of $P$ being Iwanaga-Gorenstein, and we prove Theorem \ref{theorem:I5} above.  In Section \ref{Functor categories section} we show that the functor category $\cB^{\cC}$ has a Nakayama functors relative to an adjoint pair, where $\cB$ is an abelian category and $\cC$ is locally bounded Hom-finite category, see Theorem \ref{Theorem:5}.

\subsection{Conventions}\label{Coventions}
All categories are assumed to be preadditive, and all functors are assumed to be additive. Also, all subcategories are assumed to be full. We always let $\cA$ and $\cB$ denote abelian categories. For a natural transformation $\eta\colon F\to F'$ we let $\eta_c\colon F(c)\to F(c')$ denote the component of $\eta$ at $c$. The natural transformation obtained by precomposing with a functor $G$ is denoted by $\eta_G\colon F\circ G\to F'\circ G$. Unless otherwise specified, $k$ denotes a commutative ring. We let $\proj k$ denote the category of finitely generated projective $k$-modules. For a ring $\Lambda$ we let $\Lambda\text{-}\Md$ and $\Lambda\text{-}\md$ denote the category of left $\Lambda$-modules and finitely presented left $\Lambda$-modules, respectively. 

\section{Background}\label{Preliminaries}
\subsection{Transformation of adjoints}

Let $\cD$ and $\cE$ be preadditive categories. An adjunction $(L,R,\phi,\alpha,\beta)\colon \cD\to \cE$ is given by a pair of functors $L\colon \cD\to \cE$ and $R\colon \cE\to \cD$ together with an isomorphism
\[
\phi\colon \cE(L(D),E)\to \cD(D,R(E)) 
\]
for all pairs $D\in \cD$ and $E\in \cE$, which is natural in $D$ and $E$. We also write $L\dashv R$ if there exists such an adjunction. Here $\alpha_D:=\phi(1_{L(D)})\colon D\to RL(D)$ and $\beta_E:=\phi^{-1}(1_{R(E)})\colon LR(E)\to E$ are the unit and the counit of the adjunction. They satisfy the \emphbf{triangular identities}, i.e $R(\beta_E) \circ \alpha_{R(E)}= 1_{R(E)}$ and $\beta_{L(D)}\circ L(\alpha_D)=1_{L(D)}$. By naturality we have that $\phi(f)=R(f)\circ \alpha_D$ and $\phi^{-1}(g) = \beta_E \circ L(g)$ for any morphisms $f\colon L(D)\to E$ and $g\colon D\to R(E)$. If we are working with several adjunctions, we write $\adjiso{L}{R}$, $\unit{L}{R}$ and $\counit{L}{R}$ for the isomorphism, the unit and the counit to indicate which adjunction we mean. 

Assume $(L_1,R_1,\phi_1,\alpha_1,\beta_1)\colon \cD\to \cE$ and $(L_2,R_2,\phi_2,\alpha_2,\beta_2)\colon \cD\to \cE$ are adjunctions. Following \cite{MLan98}, we say that two natural transformations \\ $\sigma\colon L_1\to L_2$ and $\tau\colon R_2\to R_1$ are \emphbf{conjugate} (for the given adjunctions) if the square
\begin{equation}\label{equation:0.1}
\begin{tikzpicture}[description/.style={fill=white,inner sep=2pt}]
\matrix(m) [matrix of math nodes,row sep=2.5em,column sep=5.0em,text height=1.5ex, text depth=0.25ex] 
{ \cE(L_2(D),E) & \cD(D,R_2(E)) \\
  \cE(L_1(D),E) & \cD(D,R_1(E)) \\};
\path[->]
(m-1-1) edge node[auto] {$\phi_2$} 	    															(m-1-2)
(m-2-1) edge node[auto] {$\phi_1$} 	    																		(m-2-2)

(m-1-1) edge node[auto] {$-\circ \sigma_D$} 	    													  (m-2-1)
(m-1-2) edge node[auto] {$\tau_E\circ -$} 	    												(m-2-2);	
\end{tikzpicture}
\end{equation}
commutes for all pairs $D\in \cD$ and $E\in \cE$.

\begin{Proposition}\label{prop:1.6}
Let $(L_i,R_i,\phi_i,\alpha_i,\beta_i)\colon \cD\to \cE$ be adjunctions for $1\leq i\leq 2$. The following hold:
\begin{enumerate}
	\item\label{prop:1.6,1}If $\sigma\colon L_1\to L_2$ is a natural transformation, then there exists a unique natural transformation $\tau\colon R_2\to R_1$ which is conjugate to $\sigma$;
	\item\label{prop:1.6,2} If $\tau\colon R_2\to R_1$ is a natural transformation, then there exists a unique natural transformation $\sigma\colon L_1\to L_2$ which is conjugate to $\tau$; 
\end{enumerate}
\end{Proposition}

\begin{proof}
This follows from \cite[Theorem IV.7.2]{MLan98}
\end{proof}

\begin{Proposition}\label{Conjugate units and counits}
Assume we have adjunctions 
\[(L,G,\phi_1,\alpha_1,\beta_1)\colon \cD\to \cE \quad \text{and} \quad (G,R,\phi_2,\alpha_2,\beta_2)\colon \cE\to \cD.
\] Then $\alpha_1$ and $\beta_2$ are conjugate, and $\beta_1$ and $\alpha_2$ are conjugate.
\end{Proposition}

\begin{proof}
This follows from \cite[Theorem 2.15]{Lau06} and its dual.
\end{proof}

We need the following result later.

\begin{Proposition}\label{prop:1.65}
Let $\cA$ and $\cB$ be abelian categories. Assume 
\[
(L_1,R_1,\phi_1,\alpha_1,\beta_1)\colon \cB\to \cA \quad \text{and} \quad (L_2,R_2,\phi_2,\alpha_2,\beta_2)\colon \cB\to \cA 
\]
are adjunctions, and $\sigma\colon L_1\to L_2$ and $\tau\colon R_2\to R_1$ are conjugate natural transformations. Then the functor $\Ker \tau\colon \cA\to \cB$ is right adjoint to the functor $\Coker \sigma\colon \cB\to \cA$.
\end{Proposition}

\begin{proof}
Let $A\in \cA$ and $B\in \cB$ be arbitrary. We have exact sequences 
\begin{align*}
L_1(B)\xrightarrow{\sigma_B}L_2(B)\to \Coker \sigma_{B}\to 0 \\
0\to \Ker \tau_{A}\to R_2(A)\xrightarrow{\tau_A}R_1(A).
\end{align*}
Applying $\cA(-,A)$ to the first sequence and $\cB(B,-)$ to the second sequence gives a diagram
\[
\begin{tikzpicture}[description/.style={fill=white,inner sep=2pt}]
\matrix(m) [matrix of math nodes,row sep=2.5em,column sep=3.0em,text height=1.5ex, text depth=0.25ex] 
{0 & \cA(\Coker \sigma_{B},A) & \cA(L_2(B),A) & \cA(L_1(B),A) \\
 0 & \cB(B,\Ker \tau_{A}) & \cB(B,R_2(A)) & \cB(B,R_1(A))  \\};
\path[->]
(m-1-1) edge node[auto] {$$} 	    												(m-1-2)	             
(m-1-2) edge node[auto] {$$} 	    												(m-1-3)	 
(m-1-3) edge node[auto] {$-\circ \sigma_{B}$} 	    								(m-1-4)
(m-2-1) edge node[auto] {$$} 	    												(m-2-2)	    
(m-2-2) edge node[auto] {$$} 	    												(m-2-3)		
(m-2-3) edge node[auto] {$\tau_{A}\circ -$} 	    						        (m-2-4)

(m-1-2) edge[dashed] node[auto] {$\cong$} 	    								    (m-2-2)	
(m-1-3) edge node[auto] {$\phi_2$} 	    									  		(m-2-3)		
(m-1-4) edge node[auto] {$\phi_1$} 	    								   			(m-2-4);	
\end{tikzpicture}
\]
with exact rows. Since $\tau$ and $\sigma$ are conjugate, the right square commutes. Therefore, we get an induced natural isomorphism 
\[
\cA(\Coker \sigma_{B},A)\cong \cB(B,\Ker \tau_{A})
\]
and the result follows.
\end{proof}

\subsection{Monads and comonads}\label{Monads and comonads}
 We refer to chapter 4 in \cite{Bor94a} and chapter VI in \cite{MLan98} for the general theory of monads and comonads. Let $\cD$ and $\cE$ be preadditive categories.

\begin{Definition}\label{def:1} 
\hspace{2em}
\begin{enumerate} 
\item A \emphbf{monad} on $\cD$ is a tuple $\sfT=(T,\eta,\mu)$, where $T\colon\cD\to \cD$ is a functor and $\eta \colon 1_{\cD}\to T$ and $\mu\colon T\circ T\to T$ are natural transformations such that the diagrams 
\begin{equation*}
\begin{tikzpicture}[description/.style={fill=white,inner sep=2pt}]
\matrix(m) [matrix of math nodes,row sep=2.5em,column sep=5.0em,text height=1.5ex, text depth=0.25ex] 
{ TTT & TT  \\
	TT & T  \\};
\path[->]
(m-1-1) edge node[auto] {$T(\mu)$} 	    												  			(m-1-2)
(m-2-1) edge node[auto] {$\mu$} 	    												  (m-2-2)

(m-1-1) edge node[auto] {$\mu_T$} 	    																	(m-2-1)
(m-1-2) edge node[auto] {$\mu$} 	    													(m-2-2);	
\end{tikzpicture}
\begin{tikzpicture}[description/.style={fill=white,inner sep=2pt}]
\matrix(m) [matrix of math nodes,row sep=2.5em,column sep=5.0em,text height=1.5ex, text depth=0.25ex] 
{ & T &   \\
  T & TT  & T \\};
\path[->]
(m-2-1) edge node[below] {$T(\eta)$} 	    												  (m-2-2)
(m-2-3) edge node[below] {$\eta_T$} 	    												(m-2-2)

(m-2-2) edge node[auto] {$\mu$} 	    													(m-1-2)

(m-2-1) edge node[above] {$1_T$} 	    													(m-1-2)
(m-2-3) edge node[above] {$1_T$} 	    													(m-1-2);	
\end{tikzpicture}
\end{equation*}
commute.

\item A \emphbf{comonad} on $\cD$ is a tuple $\sfS=(S,\epsilon,\Delta)$, where $S\colon\cD\to \cD$ is a functor and $\epsilon \colon S\to 1_{\cD}$ and $\Delta\colon S \to S \circ S$ are natural transformations such that the diagrams  
\begin{equation*}
\begin{tikzpicture}[description/.style={fill=white,inner sep=2pt}]
\matrix(m) [matrix of math nodes,row sep=2.5em,column sep=5.0em,text height=1.5ex, text depth=0.25ex] 
{ S & SS  \\
	SS & SSS  \\};
\path[->]
(m-1-1) edge node[auto] {$\Delta$} 	    												  			(m-1-2)
(m-2-1) edge node[auto] {$S(\Delta)$} 	    												  (m-2-2)

(m-1-1) edge node[auto] {$\Delta$} 	    																	(m-2-1)
(m-1-2) edge node[auto] {$\Delta_{S}$} 	    													(m-2-2);	
\end{tikzpicture}
\begin{tikzpicture}[description/.style={fill=white,inner sep=2pt}]
\matrix(m) [matrix of math nodes,row sep=2.5em,column sep=5.0em,text height=1.5ex, text depth=0.25ex] 
{ & S &   \\
  S & SS  & S \\};
\path[->]
(m-2-2) edge node[below] {$S(\epsilon)$} 	    												  (m-2-1)
(m-2-2) edge node[below] {$\epsilon_{S}$} 	    												(m-2-3)

(m-1-2) edge node[auto] {$\Delta$} 	    													(m-2-2)

(m-1-2) edge node[above] {$1_S$} 	    													(m-2-1)
(m-1-2) edge node[above] {$1_S$} 	    													(m-2-3);	
\end{tikzpicture}
\end{equation*}
commute.
\end{enumerate}
\end{Definition}
Note that a comonad on $\cD$ is the same as a monad on $\cD\op$. 

A \emphbf{morphism of monads} $\delta\colon (T_1,\eta_1,\mu_1)\to (T_2,\delta_2,\mu_2)$ is given by a natural transformation $\delta\colon T_1\to T_2$ satisfying $\mu_2\circ \delta^2 = \delta \circ \mu_1$ and $\eta_2 = \delta\circ \eta_1$, where $\delta^2 = T_2(\delta)\circ \delta_{T_1}= \delta_{T_2}\circ T_1(\delta)$. Dually, a \emphbf{morphism of comonads} $\zeta\colon (S_1,\epsilon_1,\Delta_1)\to (S_2,\epsilon_2,\Delta_2)$  is given by a natural transformation $\zeta\colon S_1\to S_2$ satisfying $\Delta_2\circ \zeta = \zeta^2\circ \Delta_1$ and $\epsilon_1 = \epsilon_2\circ \zeta$, where $\zeta^2 = S_2(\zeta)\circ \zeta_{S_1}= \zeta_{S_2}\circ S_1(\zeta)$. For a monad $\sfT=(T,\eta,\mu)$ on $\cD$, an object $D$ is called $\sfT$\emphbf{-injective} if $\eta_D\colon D\to T(D)$ is a split monomorphism. These are precisely the summands of objects of the form $T(D)$. Dually, for a comonad $\sfS=(S,\epsilon,\Delta)$ on $\cD$, an object $D$ is called $\sfS$-\emphbf{projective} if $\epsilon_D\colon S(D)\to D$ is a split epimorphism. These are precisely the summands of objects of the form $S(D)$.

\begin{Remark}\label{remark:1.5}
The category of endofunctors on $\cD$ is a monoidal category. The product is given by composition, and the unit object is the identity functor. A monad and a comonad is just a monoid and a comonoid in this monoidal category.
\end{Remark}

\begin{Example}\label{example:1}
If $(L,R,\phi,\alpha,\beta)\colon \cD\to \cE$ is an adjunction, then the tuple $\sfT=(R\circ L,\alpha, R(\beta_{L}))$ is a monad on $\cD$, and the tuple $\sfS=(L\circ R,\beta, L(\alpha_{R}))$ is a comonad on $\cE$. 
\end{Example}

We also have a notion of adjunctions between monads and comonads.

\begin{Definition}\label{adjoint monads and comonads}
Let $\sfT=(T,\eta,\mu)$ and $\sfS=(S,\epsilon,\Delta)$ be a monad and comonad on $\cD$. 
\begin{enumerate}
\item We say that $\sfS$ is \emphbf{left adjoint} to $\sfT$ (written $\sfS\dashv \sfT$) if $S$ is left adjoint to $T$ such that $\Delta\colon S\to SS$ and $\mu\colon TT\to T$ are conjugate, and $\epsilon\colon S\to 1_{\cD}$ and $\eta\colon 1_{\cD}\to T$ are conjugate.
\item We say that $\sfS$ is \emphbf{right adjoint} to $\sfT$ (written $\sfT\dashv \sfS$) if $S$ is right adjoint to $T$ such that $\Delta\colon S\to SS$ and $\mu\colon TT\to T$ are conjugate, and $\epsilon\colon S\to 1_{\cD}$ and $\eta\colon 1_{\cD}\to T$ are conjugate.
\end{enumerate} 
\end{Definition}
If $\sfT=(T,\eta,\mu)$  is a monad and $S$ is a left or right adjoint of $T$, then by Proposition \ref{prop:1.6} there exist a unique comonad $\sfS=(S,\epsilon,\Delta)$ which is left or right adjoint to $\sfT$, respectively.

We recall the Eilenberg-Moore category of a monad.

\begin{Definition}\label{Eilenberg-Moore category}
Let $\sfT=(T,\eta,\mu)$ and $\sfS=(S,\epsilon,\Delta)$ be a monad and comonad on $\cD$.
\begin{enumerate}
\item A $\sfT$\emphbf{-module} is a morphism $s\colon T(D)\to D$ satisfying $s\circ \eta_D=1_D$ and $s\circ \mu_D = s\circ T(s)$;
\item A $\sfS$\emphbf{-comodule} is a morphism $t\colon D\to S(D)$ satisfying $\epsilon_D\circ t = 1_D$ and $\Delta_D\circ t = S(t)\circ t$;
\item The \emphbf{Eilenberg-Moore category} $\cD^{\sfT}$ of $\sfT$ is the category where the objects are all $\sfT$-modules, and where a morphism between $\sfT$-modules $T(D_1)\xrightarrow{s_1} D_1$ to $T(D_2)\xrightarrow{s_2} D_2$ is a morphism $f\colon D_1\to D_2$ in $\cD$ satisfying $f\circ s_1=s_2\circ T(f)$.
\item The \emphbf{Eilenberg-Moore category} $\cD^{\sfS}$ of $\sfS$ is the category where the objects are all $\sfS$-comodules, and where a morphism between $\sfS$-comodules $D_1\xrightarrow{t_1} S(D_1)$ to $D_2\xrightarrow{t_2} S(D_2)$ is a morphism $f\colon D_1\to D_2$ in $\cD$ satisfying $t_2\circ f=S(f)\circ t_1$.
\item The \emphbf{Kleisli category} $\cD_{\sfT}$ of $\sfT$ is the category with objects being the objects in $\cD$, and with morphisms being
\[
\cD_{\sfT}(D_1,D_2):= \cD(D_1,T(D_2))
\]
Composition of $D_1\xrightarrow{g} T(D_2)$ and $D_2\xrightarrow{h} T(D_3)$ in $\cD_{\sfT}$ is given by the composite $D_1\xrightarrow{g} T(D_2) \xrightarrow{T(h)}TT(D_3)\xrightarrow{\mu_{D_3}}T(D_3)$ in $\cD$. The unit at  $D\in\cD_{\sfT}$ is $\eta_D\colon D\to T(D)$. 
\item The \emphbf{co-Kleisli category} $\cD_{\sfS}$ of $\sfS$ is the category with objects being the objects in $\cD$, and with morphisms being
\[
\cD_{\sfS}(D_1,D_2):= \cD(S(D_1),D_2)
\]
Composition of $S(D_1)\xrightarrow{g} D_2$ and $S(D_2)\xrightarrow{h} D_3$ in $\cD_{\sfS}$ is given by the composite $S(D_1)\xrightarrow{\Delta_{D_1}}SS(D_1)\xrightarrow{S(g)} S(D_2) \xrightarrow{h}D_3$ in $\cD$. The unit at  $D\in\cD_{\sfS}$ is  $\epsilon_D\colon S(D)\to D$. 
\end{enumerate}
\end{Definition} 

The forgetful functor $U^{\sfT}\colon \cD^{\sfT}\to \cD$ admits a left adjoint $F^{\sfT}\colon \cD\to \cD^{\sfT}$ sending $D$ to $TT(D)\xrightarrow{\mu_D}T(D)$ and $g\colon D_1\to D_2$ to $T(g)\colon T(D_1)\to T(D_2)$, see \cite[Proposition 4.1.4]{Bor94a}. Similarly, we have functors
\begin{align*}
& U_{\sfT}\colon \cD_{\sfT}\to \cD \quad (D\mapsto T(D)) \quad ((g\colon D_1\xrightarrow{}T(D_2))\mapsto \mu_{D_2}\circ T(g)) \\
& F_{\sfT}\colon \cD\to \cD_{\sfT} \quad (D\mapsto D) \quad ((f\colon D_1\to D_2)\mapsto \eta_{D_2}\circ f))
\end{align*}
where $F_{\sfT}$ is left adjoint to $U_{\sfT}$, see \cite[Proposition 4.1.7]{Bor94a}. For both of the adjunctions $F^{\sfT}\dashv U^{\sfT}$ and $F_{\sfT}\dashv U_{\sfT}$ the induced monad on $\cD$ is $\sfT$, see \cite[Proposition 4.2.2]{Bor94a}. The dual statements hold for the comonad $\sfS$. 

We need the following technical result to prove Theorem \ref{Isomorphism of monads adjoint pairs}, which is used to prove Theorem \ref{Admissible adjunction theorem introduction} stated in the introduction. Here, for a functor $F\colon \cD\to \cE$, we let $\im F$ denote the image of $F$. This is a category with same objects as $\cD$, and where a morphism between $D_1$ and $D_2$ is a morphism $F(D_1)\to F(D_2)$ in $\cE$.

\begin{Theorem}\label{correspondence morphism of monads and morphisms of the images}
Let 
\[
(L_1,R_1,\phi_1,\alpha_1,\beta_1)\colon \cD\to \cE \quad \text{and} \quad (L_2,R_2,\phi_2,\alpha_2,\beta_2)\colon \cD\to \cE
\]
be adjunctions, and let 
\[
\sfT_1:=(R_1\circ L_1,\alpha_1, R_1((\beta_1)_{L_1}))\quad \text{and} \quad \sfT_2:=(R_2\circ L_2,\alpha_2, R_2((\beta_2)_{L_2}))
\]
be the induced monads on $\cD$.
\begin{enumerate}
\item\label{correspondence morphism of monads and morphisms of the images:1} Given a morphism $\gamma\colon \sfT_1\to \sfT_2$ of monads, we get a functor 
\[
E_{\gamma}\colon \im L_1\to \im L_2
\]
satisfying $E_{\gamma}\circ L_1=L_2$. It acts as identity on objects, and sends a morphism $f\colon L_1(A_1)\to L_1(A_2)$ to 
\[
(\beta_2)_{L_2(A_2)}\circ L_2(\gamma_{A_2})\circ L_2R_1(f)\circ L_2((\alpha_1)_{A_1})\colon L_2(A_1)\to L_2(A_2)
\] 
\item\label{correspondence morphism of monads and morphisms of the images:2} Given a functor $E\colon \im L_1\to \im L_2$ satisfying $E\circ L_1=L_2$, there exists a unique morphism $\gamma \colon \sfT_1\to \sfT_2$ of monads such that $E_{\gamma}=E$. 
\end{enumerate}
\end{Theorem}

\begin{proof}
Let $\cD_{\sfT_1}$ and $\cD_{\sfT_2}$ denote the Kleisli categories of $\sfT_1$ and $\sfT_2$, and let $F_{\sfT_1}\colon \cD\to \cD_{\sfT_1}$ and $F_{\sfT_2}\colon \cD\to \cD_{\sfT_2}$ be the left adjoints in the Kleisli adjunctions. By \cite[Proposition 4.2.1]{Bor94a} there exists isomorphisms
\[
\Phi_1\colon \im L_1\to \cD_{\sfT_1} \quad \text{and} \quad \Phi_2\colon \im L_2\to \cD_{\sfT_2}
\]
satisfying $\Phi_1\circ L_1= F_{\sfT_1}$ and $\Phi_2\circ L_2=F_{\sfT_1}$. Explicitly, they act as identity on objects, and they send morphisms $f_1\colon L_1(A_1)\to L_1(A_2)$ and $f_2\colon L_2(A_1)\to L_2(A_2)$ to $\phi_1(f_1)\colon A_1\to R_1L_1(A_2)$ and $\phi_2(f_2)\colon A_1\to R_2L_2(A_2)$, respectively.  Now any morphism $\gamma\colon \sfT_1\to \sfT_2$ of monads induces a functor \\ $E'_{\gamma}\colon \cD_{\sfT_1} \to \cD_{\sfT_2}$ satisfying $E'_{\gamma}\circ F_{\sfT_1}= F_{\sfT_2}$. It acts as identity on objects and sends a morphism $A_1\xrightarrow{f} R_1L_1(A_1)$ to $A_1\xrightarrow{f} R_1L_1(A_1)\xrightarrow{\gamma_{A_1}}R_2L_2(A_1)$. The composite $E_{\gamma}=(\Phi_2)^{-1}\circ E'_{\gamma}\circ \Phi_1$ gives the functor in part $\ref{correspondence morphism of monads and morphisms of the images:1}$. 

For Part $\ref{correspondence morphism of monads and morphisms of the images:2}$, we note that by \cite[Proposition 2.4]{CVST10} (with $J$ being the identity functor in their notation) the association $\gamma \to E'_{\gamma}$ gives a one-to-one correspondence between morphisms of monads $\gamma\colon \sfT_1\to \sfT_2$ and functors $E'\colon \cD_{\sfT_1} \to \cD_{\sfT_2}$ satisfying $E'\circ F_{\sfT_1}= F_{\sfT_2}$. In particular, given a functor $E\colon \im L_1\to \im L_2$ satisfying $E\circ L_1=L_2$, the functor $E':=\Phi_2\circ E\circ \Phi_1^{-1}$ satisfies $E'\circ F_{\sfT_1}= F_{\sfT_2}$ and therefore must be equal to $E'_{\gamma}$ for some $\gamma\colon \sfT_1\to \sfT_2$.  Composing $E'$ with the isomorphisms $\Phi_1$ and $\Phi_2$ then gives the result.
\end{proof}
For a functor $E\colon \im L_1\to \im L_2$ satisfying $E\circ L_1=L_2$, one can easily check that the unique morphism of monads $\gamma\colon \sfT_1\to \sfT_2$ with $E_{\gamma}=E$ is given by
\begin{align}\label{Formula for monad isomorphism}
\gamma := R_1L_1\xrightarrow{(\alpha_2)_{R_1L_1}} R_2L_2 R_1 L_1= R_2  E  L_1 R_1 L_1 \xrightarrow{R_2E(\beta_1)_{L_1}} R_2 E L_1=R_2 L_2.
\end{align}

A functor $R\colon \cD\to \cE$ is called \emphbf{monadic} if there exists a monad $\sfT$ on $\cE$, an equivalence $J\colon \cD\to \cE^{\sfT}$, and an isomorphism of functors $R\cong U^{\sfT}\circ J$ where $U^{\sfT}\colon \cE^{\sfT}\to \cE$ is the forgetful functor. In the following we state a weak version of Beck's monadicity theorem, a result which characterizes when a functor is monad. See \cite[Theorem 4.4.4]{Bor94a} for the general version. 

\begin{Theorem}\label{Weak Beck's monadicity}
Let $R\colon \cA\to \cB$ be a functor between abelian categories. Then $R$ is monadic if the following conditions hold:
\begin{enumerate}
\item $R$ has a left adjoint;
\item $R$ is faithful;
\item $R$ is right exact.
\end{enumerate}
\end{Theorem}

\begin{proof}
This follows from \cite[Theorem 4.4.4]{Bor94a}.
\end{proof}

\subsection{Properties of subcategories}\label{Properties of subcategories}

Fix abelian categories $\cA$ and $\cB$.

\begin{Definition}\label{Generating and cogenerating}
Let $\cX\subset \cA$ be a subcategory.
\begin{enumerate}
\item $\cX$ is \emphbf{generating} if for all objects $A\in \cA$ there exists $X\in \cX$ and an epimorphism $X\to A$;
\item $\cX$ is \emphbf{cogenerating} if for all objects $A\in \cA$ there exists $X\in \cX$ and a monomorphism $A\to X$;
\item $\cX$ is \emphbf{resolving} if it is a generating subcategory which is closed under direct summands, extensions, and kernel of epimorphism; 
\item $\cX$ is \emphbf{coresolving} if it is a cogenerating subcategory which is closed under direct summands, extensions, and cokernels of monomorphism. 
\end{enumerate}
\end{Definition}
Here we follow the conventions in \cite{Sto14} for the definition of resolving and coresolving. We can define the \emphbf{resolution dimension}  $\dim_{\cX}(A)$ of any object $A\in \cA$ with respect to a resolving subcategory $\cX$. Explicitly, it is the smallest integer $n\geq 0$ such that there exists a long exact sequence 
\[
0\to X_n\to \cdots X_1\to X_0\to A\to 0 
\] 
where $X_i\in \cX$ for $0\leq i\leq n$. We write $\dim_{\cX}(A)=\infty$ if there doesn't exists such a resolution. If $\dim_{\cX}(A)=n$, the dimension can be computed using any resolution of $A$ by objects in $\cX$, i.e for any exact sequence
\[
0\to X'_n\to \cdots X'_1\to X'_0\to A\to 0 
\]
with $X_i'\in \cX$ for all $0\leq i\leq n-1$ we get that $X'_n\in \cX$, see \cite[Proposition 2.3]{Sto14}. The \emphbf{global resolution dimension} $\dim_{\cX}(\cA)$ of $\cA$ is defined as the supremum of $\dim_{\cX}(A)$ over all $A\in \cA$. We can define the \emphbf{coresolution dimension} $\dim_{\cY}$ with respect to a coresolving category $\cY$ dually. 

\begin{Definition}[III.6.3 in \cite{GM03}]\label{Adaptable subcategories}
Let $F\colon \cA\to \cB$ be an additive functor, and let $\cX$ be a subcategory of $\cA$.
\begin{enumerate}
\item If $F$ is left exact, we say that $\cX$ is \emphbf{adapted} to $F$ if $\cX$ is cogenerating and for any exact sequence $0\to X_0\to X_{-1}\to X_{-2}\to \cdots$ with $X_i\in \cX$ for all $i\leq 0$ we have that $0\to F(X_0)\to F(X_{-1})\to F(X_{-2})\to \cdots$ is exact;
\item If $F$ is right exact, we say that $\cX$ is \emphbf{adapted} to $F$ if $\cX$ is generating and for any exact sequence $\cdots\to X_2\to X_1\to X_0\to 0$ with $X_i\in \cX$ for all $i\geq 0$ we have that $\cdots \to F(X_2)\to F(X_1)\to F(X_0)\to 0$ is exact.
\end{enumerate}
\end{Definition}

\begin{Theorem}[Theorem III.6.8 in \cite{GM03}]
Let $F\colon \cA\to \cB$ be an additive functor, and let $\cX$ be a subcategory of $\cA$. The following holds:
\begin{enumerate}
\item If $F$ is left exact and $\cX$ is adapted to $F$, then the right derived functor $RF\colon D^+(\cA)\to D^+(\cB)$ exists, and $RF(X)\cong X$ for all $X\in \cX$;
\item If $F$ is right exact and $\cX$ is adapted to $F$, then the left derived functor $LF\colon D^-(\cA)\to D^-(\cB)$ exists, and $LF(X)\cong X$ for all $X\in \cX$.
\end{enumerate}
\end{Theorem}

The $i$th right and left derived functor is define to be $R^iF:=H^{i}(RF)$ and $L_iF:=H^{-i}(LF)$ , respectively. Note that $R^iF(X)=0$ and $L_iF(X)=0$ for $X\in \cX$ and $i>0$ in these cases.

\section{Generalization of the Nakayama functor}\label{Nakayama functors section}

\subsection{Derived functors without enough projectives or injectives}\label{Derived functors without enough projectives or injectives}

Our goal is to find easy criteria for when a subcategory is adapted to a functor when the subcategory is the image of another functor. We fix a preadditive category $\cD$ and abelian categories $\cA$ and $\cB$. By abuse of notation, in the following definition we identify the image $\im F$ of a functor $F\colon \cD\to \cA$ with the full subcategory of $\cA$ consisting of all objects of the form $F(D)$ for $D\in \cD$.

\begin{Definition}\label{definition:3.0}
Let $F\colon \cD\to \cA$ be a functor.
\begin{enumerate}
\item We say that $F$ is \emphbf{cogenerating} if $\im F$ is cogenerating in $\cA$;
\item We say that $F$ is \emphbf{generating} if $\im F$ is generating in $\cA$;
\item Let $G\colon \cA\to \cB$ be a left or right exact functor. We say that $F$ is \emphbf{adapted to} $G$ if $\im F$ is adapted to $G$;
\end{enumerate}
\end{Definition}
If $\sfS=(S,\epsilon,\Delta)$ is a comonad on $\cA$, then $S$ is generating if and only if $\epsilon$ is an epimorphism. Dually, if $\sfT=(T,\eta,\mu)$ is a monad on $\cA$, then $T$ is cogenerating if and only if $\eta$ is a monomorphism.

\begin{Proposition}\label{Adaptable classes for monads and comonads}
Let $G\colon \cA\to \cB$ be a functor.
\begin{enumerate}
\item\label{Adaptable classes for monads and comonads:2} Assume $G$ is left exact, $\sfT=(T,\eta,\mu)$ is a monad on $\cA$, and the functor $T\colon \cA\to \cA$ is cogenerating. If $T$ and $G\circ T$ are exact functors, then $T$ is adapted to $G$;
\item\label{Adaptable classes for monads and comonads:1} Assume $G$ is right exact, $\sfS=(S,\epsilon,\Delta)$ is comonad on $\cA$, and the functor $S\colon \cA\to \cA$ is generating. If $S$ and $G\circ S$ are exact functors, then $S$ is adapted to $G$.
\end{enumerate}
\end{Proposition}

\begin{proof}
We only show part $\ref{Adaptable classes for monads and comonads:1}$, part $\ref{Adaptable classes for monads and comonads:2}$ follows dually. For $A\in \cA$ consider the complex $C'(A):=\cdots \xrightarrow{(\delta_2)_A} S^2(A)\xrightarrow{(\delta_1)_A} S(A)\xrightarrow{(\delta_0)_A} A\to 0$
where $\delta_n = \sum_{i=0}^n (-1)^i S^{n-i}(\epsilon_{S^i})$.  By 4.2 in \cite{BB69} the complex $\cA(S(A'),C'(A))$ is acyclic for all objects $A'$. Since $S$ is generating, the subcategory of objects of the form $S(A')$ is generating. Hence, any sequence $A_1\to A_2\to A_3$ is exact if the sequence $\cA(S('A),A_1)\to \cA(S('A),A_2)\to \cA(S('A),A_3)$ is exact for all $A'\in \cA$. In particular, the complex $C'(A)$ is itself acyclic. 

 Let $C(A):=G(C'(A))$. We prove the following claim: Let $0\to A_3\xrightarrow{f_2} A_2\xrightarrow{f_1} A_1\to 0$ be an exact sequence, and assume the complexes $C(A_1)$ and $C(A_2)$ are acyclic. Then the following holds:
\begin{enumerate}
\item\label{property:2} The sequence $0\to G(A_3)\xrightarrow{G(f_2)} G(A_2)\xrightarrow{G(f_1)} G(A_1)\to 0$ is exact;
\item\label{property:1} The complex $C(A_3)$ is acyclic.
\end{enumerate} 
To show this, we consider the commutative diagram
\[
\begin{tikzpicture}[description/.style={fill=white,inner sep=2pt}]
\matrix(m) [matrix of math nodes,row sep=2.5em,column sep=5.0em,text height=1.5ex, text depth=0.25ex] 
{ GS^2(A_3) &  GS^2(A_2)   & GS^2(A_1)   \\
 GS(A_3) &  GS(A_2)   & GS(A_1)  \\ };
\path[->]

(m-1-1) edge node[auto] {$GS^2(f_2)$} 	    										(m-1-2)
(m-1-2) edge node[auto] {$GS^2(f_1)$} 	    										(m-1-3)

(m-2-1) edge node[below] {$GS(f_2)$} 	    										(m-2-2)
(m-2-2) edge node[below] {$GS(f_1)$} 	    										(m-2-3)

(m-1-1) edge node[auto] {$G((\delta_1)_{A_3})$} 	    		    (m-2-1)
(m-1-2) edge node[auto] {$G((\delta_1)_{A_2})$} 	    			(m-2-2)
(m-1-3) edge node[auto] {$G((\delta_1)_{A_1})$} 	    			(m-2-3);	
\end{tikzpicture}
\]
with exact rows. Since $GS^3(f_1)\colon GS^3(A_2)\to GS^3(A_1)$ is an epimorphism, and $C(A_1)$ and $C(A_2)$ are acyclic, it follows that the induced map 
\[
\Ker G((\delta_1)_{A_2})\to \Ker G((\delta_1)_{A_1})
\] is an epimorphism. Hence, applying the snake lemma and using that the sequence $GS^2(A_3)\to GS(A_3)\to G(A_3)\to 0$ is exact since $G$ is right exact, we get that the sequence $0\to G(A_3)\xrightarrow{G(f_2)} G(A_2)\xrightarrow{G(f_1)} G(A_1)\to 0$ is exact. This proves part \ref{property:2} of the claim. Also it follows that we have a short exact sequences 
\[
0\to C(A_3)\xrightarrow{C(f_2)} C(A_2)\xrightarrow{C(f_1)} C(A_1)\to 0
\]
of complexes. Since two of the complexes are acyclic, the third one is also acyclic. This proves part \ref{property:1} of the claim.

To prove the result in the proposition, it is sufficient to show that for an exact sequence
\[
0\to K\to A_{n} \xrightarrow{f_n} \cdots \xrightarrow{f_2} A_1\xrightarrow{f_1} A_0\to 0
\]   
with $A_i$ being $\sfS$-projective for all $i$, we have that 
\begin{align}\label{What we want to show is exact}
0\to G(K)\to G(A_{n}) \xrightarrow{G(f_n)} \cdots \xrightarrow{G(f_2)} G(A_1)\xrightarrow{G(f_1)} G(A_0)\to 0
\end{align}
is exact. To this end, note that since $A_i$ is $\sfS$-projective, the complex $C'(A_i)$ is contractible by the proof of Proposition 8.6.8 in \cite{Wei94} (the proposition only says that $C'(A_i)$ is acyclic, but in the proof they actually show that it is contractible). Hence, the complex $C(A_i)$ is also contractible and therefore acyclic. By part \ref{property:1} of the claim it follows that the complex $C(\Ker f_i)$ is acyclic for all $i$. But then by part \ref{property:2} of the claim it follows that the sequence \eqref{What we want to show is exact} is exact. This proves the result.
\end{proof}

\begin{Proposition}\label{Adaptable classes for adjoints}
Let $G\colon \cA\to \cB$ be a functor. 
\begin{enumerate}
\item\label{Adaptable classes for adjoints:1} Assume $G$ is left exact, $(L,R,\phi,\alpha,\beta)\colon \cA\to \cD$ is an adjunction, and $R$ is cogenerating. If $R\circ L$ and $G\circ R\circ L$ are exact functors, then $R$ is adapted to $G$;
\item\label{Adaptable classes for adjoints:2} Assume $G$ is right exact, $(L',R',\phi',\alpha',\beta')\colon \cD\to \cA$ is an adjunction, and $L'$ is generating. If $L'\circ R'$ and $G\circ L'\circ R'$ are exact functors, then $L'$ is adapted to $G$.
\end{enumerate}
\end{Proposition}

\begin{proof}
From the triangle identities we have that $R(A)$ is a summand of $RLR(A)$. Hence, if $R$ is cogenerating, then $R\circ L$ is cogenerating. Therefore, by Proposition \ref{Adaptable classes for monads and comonads} applied to the induced monad $\sfT=(R\circ L,\alpha, R(\beta_{L}))$, we get that $R\circ L$ is adapted to $G$. Finally, since $R(A)$ is a summand of $RLR(A)$ for any $A\in \cA$, it follows that $R$ itself is adapted to $G$. Part $\ref{Adaptable classes for adjoints:2}$ is proved dually. 
\end{proof}

\subsection{Definition of Nakayama functors}\label{Nakayama functors subsection}

Let $\Lambda$ be a finite-dimensional algebra over a field $k$, and consider the restriction functor 
\[
f^*:= \res^{\Lambda}_k\colon \Lambda\text{-}\md\to k\text{-}\md.
\] 
It has a left adjoint $f_!:=(\Lambda\otimes_k-)\colon k\text{-}\md\to \Lambda\text{-}\md$, and it turns out that the adjoint pair $f_!\dashv f^*$ interact nicely with the Nakayama functor $\nu=\Hom_k(\Lambda,k)\otimes_{\Lambda}-\colon \Lambda\text{-}\md\to \Lambda\text{-}\md$. More precisely, the composite $\nu\circ f_!$ is right adjoint to $f^*$, and there exists an isomorphism $\nu^-\circ \nu\circ f_!\cong f_!$ where $\nu^-$ is the right adjoint of $\nu$. We generalize this in the following definition. 

\begin{Definition}\label{Nakayama functor for adjoint pair}
Let $\cD$ be an additive category and $\cA$ an abelian category, and let $f^*\colon \cA\to \cD$ be a faithful functor with left adjoint $f_!\colon \cD\to \cA$. A \emphbf{Nakayama functor} relative to $f_!\dashv f^*$ is a functor $\nu\colon \mathcal{A}\to \mathcal{A}$ with  a right adjoint $\nu^-$ satisfying:
\begin{enumerate}
\item $\nu\circ f_!$ is right adjoint to $f^*$;
\item The unit of $\nu\dashv\nu^-$ induces an isomorphism $f_!\xrightarrow{\cong} \nu^-\circ \nu \circ f_!$ when precomposed with $f_!$.
\end{enumerate}
\end{Definition}

We also say that $f_!\dashv f^*$ has a Nakayama functor $\nu$. In this case we let $f_*:=\nu\circ f_!$ denote the right adjoint. Since the left adjoint of a functor is unique up to isomorphism, the definition of the Nakayama functor only depends on the functor $f^*$ in the adjunction $f_!\dashv f^*$. It also turns out that the Nakayama functor is unique up to precomposing with an equivalence $\Phi\colon \cA\to \cA$ satisfying $\Phi\circ f_!=f_!$, see Theorem \ref{Uniqueness Nakayama functor adjoint pairs}. Note that for a finite-dimensional algebra, the classical Nakayama functor is a Nakayama functor relative to the adjoint pair $(\Lambda\otimes_k-)\dashv \res^{\Lambda}_k$ as described above. 

We fix $(\nu,\nu^- ,\theta,\lambda,\sigma )\colon \cA\to \cA$ to be the adjunction of the Nakayama functor. The following holds for adjoint pairs with Nakayama functor.

\begin{Lemma}\label{Basic properties of Nakayama functor adjoint pairs}
Assume $f_!\dashv f^*$ admits a Nakayama functor $\nu$. The following holds:
\begin{enumerate}
\item\label{generating functor} $f_!$ is generating;
\item\label{cogenerating functor} $f_*$ is cogenerating;
\item\label{Isomorphism Nakayama functor} The restriction $\nu\colon \im f_!\to \im f_*$ is an isomorphism of categories;
\item\label{Counit isomorphism Nakayama functor} $\sigma_{f_*}\colon \nu \circ \nu^{-}\circ f_*\to f_*$ is an isomorphism;
\item\label{on adjunctions} There are adjunctions $f^*\circ \nu \dashv f_!\dashv f^*\dashv f_*\dashv f^*\circ \nu^-$
\item\label{f_! is adapted} $f_!$ is adapted to $\nu$;
\item\label{f_* is adapted} $f_*$ is adapted to $\nu^-$.
\end{enumerate} 
\end{Lemma}

\begin{proof}
Since $f^*$ is faithful, it follows by \cite[Theorem IV.3.1]{MLan98} and its dual that the counit of $f_!\dashv f^*$ is an epimorphism and the unit of $f^*\dashv f_*$ is a monomorphism. This shows that $f_!$ is generating and $f_*$ is cogenerating. Part $\ref{Isomorphism Nakayama functor}$ is obvious. For part $\ref{Counit isomorphism Nakayama functor}$, note that $\sigma_{\nu\circ f_!}\circ \nu(\lambda_{f_!})=1_{\nu\circ f_!}$ by the triangle identities for $\nu\dashv \nu^-$. Since $\lambda_{f_!}$ is an isomorphism, it follows that $\sigma_{f_*}=\sigma_{\nu\circ f_!}$ is an isomorphism. For part $\ref{on adjunctions}$, since $f^*\dashv \nu\circ f_!$, we have 
\[
f^*\circ \nu \dashv \nu^-\circ (\nu\circ f_!)\cong f_!.
\] 
Also, since $f_!\dashv f^*$, we have that $\nu\circ f_!\dashv f^*\circ \nu^-$. Part $\ref{f_! is adapted}$ and $\ref{f_* is adapted}$ follows immediately from Proposition $\ref{Adaptable classes for adjoints}$.
\end{proof}

Note that if $\nu$ is a Nakayama functor relative to $f_!\dashv f^*$, then $\nu^-$ is a Nakayama functor relative to $f_*\dashv f^*$ in the opposite categories.

For the constructions and results in Section \ref{Gorenstein categories for comonads} we only need to work with the endofunctor $P:=f_!\circ f^*$, and not with the adjoint pair. We therefore introduce the following definition.

\begin{Definition}\label{definition:N1}
Let $\cA$ be an abelian category with a generating functor $P\colon \cA\to \cA$. A \emphbf{Nakayama functor} relative to $P$ is a functor $\nu\colon \cA\to \cA$ with a right adjoint $\nu^-\colon \cA\to \cA$ satisfying the following:

\begin{enumerate}
\item\label{definition:N1,1} $\nu\circ P$ is right adjoint to $P$;
\item\label{definition:N1,2} The unit of $\nu\dashv\nu^-$ induces an isomorphism $P\xrightarrow{\cong} \nu^-\circ \nu \circ P$ when precomposed with $P$.
\end{enumerate}
\end{Definition} 

We also say that $P$ has a Nakayama functor $\nu$. Note that if $\nu$ is a Nakayama functor for for $f_!\dashv f^*$, then it is a Nakayama functor for the composite $P:=f_!\circ f^*$.  

In the following we let $I=\nu\circ P$. We also fix $(\nu,\nu^- ,\theta,\lambda,\sigma )\colon \cA\to \cA$ to be the adjunction of the Nakayama functor.

\begin{Lemma}\label{lemma:N1}
Let $\cA$ be an abelian category with a generating functor $P\colon \cA\to \cA$ and a Nakayama functor $\nu$ relative to $P$. The following holds:

\begin{enumerate}
	\item\label{lemma:N1,1} The restriction $\nu\colon \im P\to \im I$ is an isomorphism of categories;
	\item\label{lemma:N1,2} $\sigma_I\colon \nu \circ \nu^{-}\circ I\to I$ is an isomorphism;
	\item\label{lemma:N1,3} There are adjunctions $P\circ \nu\dashv P\dashv I\dashv I\circ \nu^{-}$;
	\item\label{lemma:N1,5} $P$ is adapted to $\nu$;
	\item\label{lemma:N1,6} $I$ is adapted to $\nu^-$;
	\item\label{lemma:N1,4} $I\colon \cA\to \cA$ is a cogenerating functor;
	\item\label{lemma:N1,7} $P$ is faithful;
	\item\label{lemma:N1,8} $I$ is faithful.
\end{enumerate}

\end{Lemma}

\begin{proof}
Part $\ref{lemma:N1,1}$-$\ref{lemma:N1,6}$ is proved similarly to part $\ref{Isomorphism Nakayama functor}$-$\ref{f_* is adapted}$ in Lemma \ref{Basic properties of Nakayama functor adjoint pairs}. For part $\ref{lemma:N1,4}$ note that the functor $\Ker \unit{I}{I\circ \nu^-}$ is right adjoint to the functor $ \Coker \counit{P}{I}$ by Proposition \ref{Conjugate units and counits} and Proposition \ref{prop:1.65}. Since $P$ is generating, we have $\Coker \counit{P}{I}=0$, and therefore $\Ker \unit{I}{I\circ \nu^-}=0$. Hence, $I$ is cogenerating. For part $\ref{lemma:N1,7}$, note that we have an equality $\unit{P}{I}_B\circ f = IP(f)\circ \unit{P}{I}_A$ for any morphism $f\colon A\to B$. Since the functor $I$ is cogenerating, the unit $\unit{P}{I}_B$ of the adjunction $P\dashv I$ is a monomorphism. Hence, the left hand side of the equality is nonzero. Therefore $IP(f)\circ \unit{P}{I}_A\neq 0$, hence $IP(f)\neq 0$, and so $P(f)\neq 0$. This shows that $P$ is faithful. Part $\ref{lemma:N1,8}$ is proved dually.  
\end{proof}

Note that if $\nu$ is a Nakayama functor relative to $P$, then $\nu^-$ is a Nakayama functor relative to $I$ in the opposite category.

We now give several examples of Nakayama functors for adjoints. They are all special cases of Theorem \ref{Theorem:5}, obtained by identifying with functor categories via $\Lambda\text{-}\Md=(k\text{-}\Md)^{\Lambda}$, $(\Lambda_1\otimes_k \Lambda_2)\text{-}\Md=(\Lambda_2\text{-}\Md)^{\Lambda_1}$ and $\cC\text{-}\Md=(k\text{-}\Md)^{\cC}$ in Example \ref{One algebra}, \ref{Two algebras} and \ref{Functor category example}, respectively.

\begin{Example}\label{One algebra}
Let $k$ be a commutative ring, and let $\Lambda$ be a $k$-algebra which is finitely generated and projective as a $k$-module. Furthermore, let
\begin{align*}
f^*:=\res^{\Lambda}_k\colon \Lambda\text{-}\Md\to k\text{-}\Md \quad \text{and} \quad f_!:=\Lambda\otimes_k -\colon k\text{-}\Md \to \Lambda\text{-}\Md
\end{align*}   
Then the functor 
\[
\nu=\Hom_k(\Lambda,k)\otimes_{\Lambda}-\colon \Lambda\text{-}\Md\to \Lambda\text{-}\Md
\]
is a Nakayama functor relative to the adjoint pair $f_!\dashv f^*$. This follows from Theorem \ref{Theorem:5}, but for the convenience of the reader we explain the argument in more detail here: First note that the composite $\nu\circ f_!$ is given by 
\[
\nu\circ f_!\cong \Hom_k(\Lambda,k)\otimes_k -\colon k\text{-}\Md \to \Lambda\text{-}\Md
\]
Since $\Lambda$ is finitely generated projective over $k$, the functor 
\[
\Hom_k(\Lambda,-)\colon k\text{-}\Md \to \Lambda\text{-}\Md
\]
 is exact and preserves coproducts, and hence there exists an isomorphism 
\[
\Hom_{k}(\Lambda,-)\cong \Hom_k(\Lambda,k)\otimes_k -\colon k\text{-}\Md \to \Lambda\text{-}\Md
\]
Therefore, $\nu\circ f_!$ is right adjoint to $f^*$. Also, the right adjoint of $\nu$ is given by
\[
\nu^-=\Hom_{\Lambda}(\Hom_k(\Lambda,k),-)\colon \Lambda\text{-}\Md\to \Lambda\text{-}\Md
\] 
and we have
\[
\nu^-\circ \nu\circ f_! \cong \Hom_{\Lambda}(\Hom_k(\Lambda,k),\Hom_{k}(\Lambda,-))  \cong \Hom_k(\Hom_k(\Lambda,k),-)   
\]
where the second isomorphism follows from the Hom-tensor adjunction. Furthermore, since $\Hom_k(\Lambda,k)$ is finitely generated projective over $k$, we have an isomorphism
\[
\Hom_k(\Hom_k(\Lambda,k),-)\cong \Hom_k(\Hom_k(\Lambda,k),k)\otimes_k-
\]
Finally, since $\Lambda$ is finitely generated projective over $k$, the natural map $\Lambda \to \Hom_k(\Hom_k(\Lambda,k),k)$ is an isomorphism, and hence we have 
\[
\Hom_k(\Hom_k(\Lambda,k),k)\otimes_k-\cong \Lambda\otimes_k - = f_!.
\]
Composing these maps, we get an isomorphism $f_!\cong \nu^-\circ \nu\circ f_!$ which is given by the unit of the adjunction $\nu\dashv \nu^-$.

  If we assume $k$ is coherent, then $f_!\dashv f^*$ restrict to an adjoint pair with Nakayama functor $\nu$ between the category of finitely presented modules $\Lambda\text{-}\md$ and $k\text{-}\md$. Note that if $k$ is a field, then we just obtain the classical Nakayama functor.
\end{Example}

\begin{Example}\label{Two algebras}
Let $k$ be a commutative ring, let $\Lambda_1$ be a $k$-algebra which is finitely generated and projective as a $k$-module, and let $\Lambda_2$ be a $k$-algebra. Furthermore, let
\begin{align*}
& f^*:=\res^{\Lambda_1\otimes_k \Lambda_2}_{\Lambda_2}\colon \Lambda_1\otimes_k \Lambda_2\text{-}\Md\to \Lambda_2\text{-}\Md \\
& f_!:=\Lambda_1\otimes_k -\colon \Lambda_2\text{-}\Md \to \Lambda_1\otimes_k \Lambda_2\text{-}\Md
\end{align*}   
Then the functor 
\[
\nu=\Hom_k(\Lambda_1,k)\otimes_{\Lambda_1}-\colon \Lambda_1\otimes_k \Lambda_2\text{-}\Md\to \Lambda_1\otimes_k \Lambda_2\text{-}\Md
\]
is a Nakayama functor for the adjoint pair $f_!\dashv f^*$. If we furthermore assume $\Lambda_2$ is left coherent, then $f_!\dashv f^*$ restrict to an adjoint pair with Nakayama functor $\nu$ between the category of finitely presented modules $\Lambda_1\otimes_k \Lambda_2\text{-}\md$ and $\Lambda_2\text{-}\md$. 
\end{Example}

\begin{Example}\label{Functor category example}
Let $k$ be a commutative ring, and let $\cC$ be a small, $k$-linear, Hom-finite and locally bounded category, see Definition \ref{Definition:12,5}. Then the evaluation functor 
\[
i^*\colon \cC\text{-}\Md\to \prod_{c\in \cC}k\text{-}\Md \quad \quad i^*(F)= (F(c))_{c\in \cC} \\
\] 
has a left adjoint $i_!$ given by 
\[
i_!\colon \prod_{c\in \cC}k\text{-}\Md \to \cC\text{-}\Md \quad \quad i_!(M_c)_{c\in \cC} = \bigoplus_{c\in \cC}\cC(c,-)\otimes_k M_c 
\]
Furthermore, we have a functor
\[
\nu= \Hom_k(\cC,k)\otimes_{\cC}-\colon \cC\text{-}\Md\to\cC\text{-}\Md  
\]
given by $\nu(M)(c)=\Hom_k(\cC(c,-),k)\otimes_{\cC}M$, and it turns out that this is a Nakayama functor relative to $i_!\dashv i^*$. See Subsection \ref{Comonad with Nakayama functor on} for more details. 
\end{Example}

\begin{Example}\label{Complexes}
Let $k$ be a commutative ring, and let $\cC$ be the $k$-linear category generated by the quiver 
\[
\cdots \xleftarrow{d_{i-1}} c_{i-1}\xleftarrow{d_i} c_{i}\xleftarrow{d_{i+1}} \cdots
\]
with vertex set $\{c_i|i\in \bZ/n \bZ\}$ and relations $d_i\circ d_{i+1}=0$. The category $\cB^{\cC}$ of $k$-linear functors from $\cC$ to $\cB$ can be identified with $n$-periodic complexes $B_{\bullet}= \cdots \xleftarrow{d_{i-1}} B_{i-1}\xleftarrow{d_i} B_{i}\xleftarrow{d_{i+1}} \cdots$ over $\cB$ (for $n=0$ this is just unbounded complexes over $\cB$). Let $[1]\colon \cB^{\cC}\to \cB^{\cC}$ be the shift functor defined by $(B_{\bullet}[1])_i=B_{i+1}$. The restriction functor 
\[
i^*\colon \cB^{\cC}\to \prod_{i\in \bZ/n \bZ}\cB \quad \quad i^*(B_{\bullet})=(B_i)_{i\in \bZ/n \bZ}
\]
has a left adjoint
\[
i_!\colon \prod_{i\in \bZ/n \bZ}\cB\to \cB^{\cC} \quad \quad i_!(B_i)_{i\in \bZ/n \bZ}=\bigoplus_{i\in \bZ/n \bZ} C(B_i)[-i]
\]
and a right adjoint
\[
i_*\colon \prod_{i\in \bZ/n \bZ}\cB\to \cB^{\cC} \quad \quad i_*(B_i)_{i\in \bZ/n \bZ}=\bigoplus_{i\in \bZ/n \bZ} C(B_i)[-i-1]
\]
where $C(B_i)$ is the complex with $C(B_i)_0=C(B_i)_{-1}=B_i$, the differential $C(B_i)_0\to C(B_i)_{-1}$ is the identity, and $C(B_i)_k=0$ for $k\neq 0,-1$. It follows that $[-1]$ is a Nakayama functor relative to $i_!\dashv i^*$.
\end{Example}

\begin{Example}\label{Morphism category}
Let $\bA_2=(1 \to 2)$ be the category with two objects and one morphism between them, and let $\cC=k\bA_2$ be the $k$-linearization of $\bA_2$. An object in $\cB^{k\bA_2}$ is then just a morphism $B_1\xrightarrow{f}B_2$ in $\cB$. The restriction functor
\[
i^*\colon \cB^{k\bA_2}\to \cB \prod \cB \quad \quad i^*(B_1\xrightarrow{f}B_2)=(B_1,B_2)
\]
has a left adjoint given by
\[
i_!\colon  \cB \prod \cB\to \cB^{k\bA_2} \quad \quad i_!(B_1,B_2)=(B_1\xrightarrow{1}B_1) \oplus (0\to B_2).
\]
The cokernel functor $\nu(B_1\xrightarrow{f}B_2)= B_2\to \Coker f$ is then a Nakayama functor relative to $i_!\dashv i^*$, and its right adjoint $\nu^-$ is given by $\nu^-(B_1\xrightarrow{f}B_2)= \Ker f\to B_1$. 

More generally, for $\bA_n=(1\to 2\to \cdots \to n)$ the restriction functor
\[
i^*\colon \cB^{k\bA_n}\to \prod_{1\leq i\leq n} \cB \quad \quad i^*(B_1\xrightarrow{f_1}B_2\xrightarrow{f_2}\cdots \xrightarrow{f_{n-1}}B_n)=(B_1,B_2,\cdots ,B_n)
\]
has a left adjoint $i_!\colon  \prod_{1\leq i\leq n} \cB\to \cB^{k\bA_n}$ given by
\begin{multline*} 
i_!(B_1,B_2,\cdots ,B_n)=\bigoplus_{1\leq i\leq n}(0\to \cdots \to 0\to B_i\xrightarrow{1}B_i\xrightarrow{1}\cdots \xrightarrow{1}B_i)
\end{multline*}
where the first $n+1-i$ terms of $(0\to \cdots \to 0\to B_i\xrightarrow{1}B_i\xrightarrow{1}\cdots \xrightarrow{1}B_i)$ are $0$. The functor
\begin{multline*}
\nu(B_1\xrightarrow{f_1}B_2\xrightarrow{f_2}\cdots \xrightarrow{f_{n-1}}B_n)= \\ 
(B_n\to \Coker f_{n-1}\to \Coker (f_{n-1} f_{n-2})\to \cdots \to \Coker (f_{n-1} f_{n-2} \cdots f_1))
\end{multline*}
is then a Nakayama functor relative to $i_!\dashv i^*$, and its right adjoint is given by
\begin{multline*}
\nu^-(B_1\xrightarrow{f_1}B_2\xrightarrow{f_2}\cdots \xrightarrow{f_{n-1}}B_n)= \\ 
(\Ker (f_{n-1} f_{n-2} \cdots f_1) \to \cdots \to \Ker (f_2f_1)\to \Ker f_1\to B_1) 
\end{multline*}
\end{Example}

\subsection{Properties of Nakayama functors}\label{Uniqueness of Nakayama functor section}
Fix an abelian category $\cA$ and an additive category $\cD$.  Our main goal in this subsection is to prove Theorem \ref{Admissible adjunction theorem introduction} which was stated in the introduction.

If $\nu$ is a Nakayama functor relative to $f_!\dashv f^*$, then the restriction 
\[
E:=\nu|_{\im f_1}\colon \im f_!\to \im f_*
\] 
is an isomorphism of categories satisfying $E\circ f_!=f_*$, where $f_*=\nu\circ f_!$. We want to show that such an isomorphism is enough to have a Nakayama functor. To this end, we assume we have adjunctions $f_!\dashv f^*\dashv f_*$. We define functors 
\begin{align*}
& \mor_{f_!}\colon \cA \to \Mor(\im f_!) \quad A\to \mor_{f_!}(A)\\
& \mor_{f_*}\colon \cA \to \Mor(\im f_*) \quad A\to \mor_{f_*}(A)
\end{align*}
where $\Mor(\im f_!)$ and $\Mor(\im f_*)$ are the categories of morphisms in $\im f_!$ and $\im f_*$, respectively. Explicitly, $\mor_{f_!}$ and $\mor_{f_*}$ are defined by 
\begin{align*}
& \mor_{f_!}(A):= i\circ \counit{f_!}{f^*}_{\Ker \counit{f_!}{f^*}_{A}}\colon f_!f^*(\Ker \counit{f_!}{f^*}_{A})\to f_!f^*(A) \\
& \mor_{f_*}(A):= \unit{f^*}{f_*}_{\Coker \unit{f^*}{f_*}_{A}}\circ p\colon f_*f^*(A)\to f_*f^*(\Coker \unit{f^*}{f_*}_{A})
\end{align*}
for an object $A\in \cA$, where $i\colon \Ker \counit{f_!}{f^*}_{A} \to f_!f^*(A)$ is the inclusion and $p\colon f_*f^*(A) \to \Coker \unit{f^*}{f_*}_{A}$ is the projection. 

\begin{Proposition}\label{Extending Nakayama functor Adjoint pair}
Assume we have adjunctions $f_!\dashv f^*\dashv f_*$ where \\ $f^*\colon \cA\to \cD$ is a faithful functor.  Furthermore, let $E\colon \im f_!\to \im f_*$ be an isomorphism satisfying $E\circ f_!=f_*$. The following holds:
\begin{enumerate}
\item There exists a Nakayama functor $\nu$ relative to $f_!\dashv f^*$ which satisfies $\nu|_{\im f_!}=E$;
\item If $\nu'$ is any other right exact functor satisfying $\nu'|_{\im f_!}=E$, then there exists a unique isomorphism $\zeta\colon \nu\to \nu'$ satisfying $\zeta_{f_!}=1_{f_*}$.
\end{enumerate} 
\end{Proposition}

\begin{proof}
Let $E\colon \im f_!\to \im f_*$ be an isomorphism satisfying $E\circ f_!=f_*$, and let $E^{-}$ be its inverse. These lifts to isomorphisms $E\colon \Mor(\im f_!) \to \Mor(\im f_*)$ and $E^{-}\colon \Mor(\im f_*) \to \Mor(\im f_!)$ defined pointwise. Note that the sequences
\begin{align*}
 & f_*f^*(\Ker \counit{f_!}{f^*}_{f_!(D)})\xrightarrow{E(\mor_{f_!}(f_!(D)))}f_*f^*f_!(D)\xrightarrow{E(\counit{f_!}{f^*}_{f_!(D)})} f_*(D)\to 0 \\
&  0\to f_!(D) \xrightarrow{E^-(\unit{f^*}{f_*}_{f_*(D)})} f_!f^*f_*(D)\xrightarrow{E^-(\mor_{f_*}(f_*(D)))} f_!f^*(\Coker \unit{f^*}{f_*}_{f_*(D)})
\end{align*}
are exact in $\cA$. Hence, we can define functors
\begin{align*}
& \nu:= \Coker \circ E\circ \mor_{f_!}\colon \cA\to \cA \\
& \nu^-:=\Ker \circ E^-\circ \mor_{f_*}\colon \cA\to \cA
\end{align*}
such that $\nu|_{\im f_!} =E$ and $\nu^-|_{\im f_*} =E^-$, where $\Ker\colon \Mor(\im f_!)\to \cA$ and $\Coker\colon \Mor(\im f_*)\to \cA$ are the kernel and cokernel functors. We claim that $\nu \dashv \nu^-$. Let $h\colon \nu (A_1)\to A_2$ be a morphism in $\cA$. By definition, we have an exact sequence 
\[
f_*f^*(\Ker \counit{f_!}{f^*}_{A_1})\xrightarrow{E(\mor_{f_!}(A_1))} f_*f^*(A_1) \to \nu(A_1)\to 0
\] 
Consider the composite
\[
\overline{h}:=f_*f^*(A_1)\xrightarrow{}\nu(A_1)\xrightarrow{h}A_2\xrightarrow{\unit{f^*}{f_*}_{A_2}}f_*f^*(A_2)
\]
It satisfies $\overline{h}\circ E(\mor_{f_!}(A_1))=0$ and $\mor_{f_*}(A_2)\circ \overline{h}=0$. Applying $E^-$ to $\overline{h}$ therefore gives a morphism $E^-(\overline{h})\colon f_!f^*(A_1)\to f_!f^*(A_2)$ satisfying $E^-(\overline{h}) \circ \mor_{f_!}(A_1)=0$ and $E^-(\mor_{f_*}(A_2))\circ E^-(\overline{h})=0$. Since we have exact sequences 
\[
f_!f^*(\Ker \counit{f_!}{f^*}_{A_1})\xrightarrow{\mor_{f_!}(A_1)} f_!f^*(A_1)\xrightarrow{\counit{f_!}{f^*}_{A_1}} A_1\to 0
\]
and
\[
0\to \nu^-(A_2)\xrightarrow{}f_!f^*(A_2)\xrightarrow{E^-(\mor_{f_*}(A_2))} f_!f^*(\Coker \unit{f^*}{f_*}_{A_2})
\]
it follows that the morphism $E^-(\overline{h})$ induces a morphism $\theta (h)\colon A_1\to \nu^-(A_2)$. Obviously, the map $h\to \theta (h)$ is bijective and natural, and therefore we have an adjunction $\nu \dashv \nu^-$. Furthermore, under this bijection the unit  $ 1_{\cA}\to \nu^-\circ \nu$ is the identity on $\im f_!$. Hence, $\nu$ is a Nakayama functor which satisfies $\nu|_{\im f_!}=E$. The uniqueness of $\nu$ is obvious.
\end{proof}

We get the following corollary from the proof of Proposition \ref{Extending Nakayama functor Adjoint pair}, which is useful for computing the Nakayama functor and its adjoint in examples.

\begin{Corollary}\label{Computing the Nakayama functor}
Assume we have adjunctions $f_!\dashv f^*\dashv f_*$ where $f^*\colon \cA\to \cD$ is a faithful functor. The following holds: 
\begin{enumerate}
\item If $\nu\colon \cA\to \cA$ is a right exact functor such that $\nu\circ f_!=f_*$ and $\nu|_{\im f_!}\colon \im f_!\to \im f_*$ is an isomorphism, then $\nu$ is a Nakayama functor relative to $f_!\dashv f^*$;
\item If $\nu^-\colon \cA\to \cA$ is a left exact functor such that $\nu^-\circ f_*=f_!$ and $\nu^-|_{\im f_*}\colon \im f_*\to \im f_!$ is an inverse to $\nu|_{\im f_!}$, then $\nu\dashv \nu^-$;
\item The unit and counit of $\nu\dashv \nu^-$ are the unique natural transformations $\lambda \colon 1_{\cA}\to \nu^-\circ \nu$ and $\sigma\colon \nu\circ \nu^-\to 1_{\cA}$ which satisfies $\lambda_{f_!}=1_{f_!}$ and $\sigma_{f_*}=1_{f_*}$.
\end{enumerate}
\end{Corollary}

Combining Proposition \ref{Extending Nakayama functor Adjoint pair} with Theorem \ref{correspondence morphism of monads and morphisms of the images} gives the following.

\begin{Theorem}\label{Isomorphism of monads adjoint pairs}
Let $f^*\colon \cA\to \cD$ be a faithful functor, and assume there exists adjunctions $f_!\dashv f^*\dashv f_*\dashv f^{\times}$.  
\begin{enumerate}
\item\label{Isomorphism of monads adjoint pairs:1} If $\nu\colon \cA\to \cA$ is a Nakayama functor relative to $f_!\dashv f^*$ satisfying $\nu\circ f_!=f_*$, then $\gamma_{\nu}:=\chi_{f_*}\circ f^*(\lambda_{f_!})$ induces an isomorphism of monads
\[
\gamma_{\nu}\colon (f^*f_!,\unit{f_!}{f^*},f^*(\counit{f_!}{f^*}_{f_!}))\xrightarrow{\cong} (f^{\times}f_*,\unit{f_*}{f^{\times}},f^{\times}(\counit{f_*}{f^{\times}}_{f_*}))
\]
where $\chi\colon f^*\circ \nu^-\xrightarrow{\cong}f^{\times}$ is the natural transformation which is conjugate to the identity $f_*=\nu\circ f_!$; 
\item\label{Isomorphism of monads adjoint pairs:2} If $\gamma\colon (f^*f_!,\unit{f_!}{f^*},f^*(\counit{f_!}{f^*}_{f_!}))\xrightarrow{\cong} (f^{\times}f_*,\unit{f_*}{f^{\times}},f^{\times}(\counit{f_*}{f^{\times}}_{f_*}))$ is an isomorphism of monads, then there exists a Nakayama functor \\ $\nu\colon \cA\to \cA$ relative to $f_!\dashv f^*$ satisfying $\nu\circ f_!=f_*$ and $\gamma_{\nu}=\gamma$. Furthermore, if $\nu'$ is another Nakayama functor relative to $f_!\dashv f^*$ satisfying $\nu'\circ f_!=f_*$ and $\gamma_{\nu'}=\gamma$, then there exists a unique isomorphism $\zeta\colon \nu\to \nu'$ with $\zeta_{f_!}=1_{f_*}$.
\end{enumerate}
\end{Theorem}

\begin{proof}
We only need to show that the natural transformation in \eqref{Formula for monad isomorphism} coincide with the one given for $\gamma_{\nu}$ above. The rest follows from Proposition \ref{Extending Nakayama functor Adjoint pair} and Theorem \ref{correspondence morphism of monads and morphisms of the images}. By \eqref{Formula for monad isomorphism} we have that
\[
\gamma_{\nu} := f^*f_!\xrightarrow{(\unit{f_*}{f^{\times}})_{f^*f_!}} f^{\times}f_*f^*f_!= f^{\times}\nu f_!f^*f_! \xrightarrow{f^{\times}\nu (\counit{f_!}{f^*})_{f_!}} f^{\times}\nu f_!=f^{\times}f_*
\]
Since 
\[
\adjiso{f_*}{f^{\times}}\colon \cA(\nu\circ f_!,-) \xrightarrow{\theta}\cA(f_!,\nu^-)\xrightarrow{\adjiso{f_!}{f^*}} \cD(-,f^*\circ \nu^- )\xrightarrow{\chi\circ -} \cD(-,f^{\times}).
\]
it follows that 
\begin{align*}
\unit{f_*}{f^{\times}} & = \adjiso{f_*}{f^{\times}}(1_{f_*})=\chi_{f_*} \circ \adjiso{f_!}{f^*}(\theta (1_{\nu\circ f_!})) = \chi_{f_*}\circ \adjiso{f_!}{f^*}(\lambda_{f_!}) \\
& = \chi_{f_*}\circ  f^*(\lambda_{f_!})\circ \unit{f_!}{f^*}
\end{align*}
Hence, it follows that 
\begin{align*}
\gamma_{\nu} & = f^{\times}\nu (\counit{f_!}{f^*})_{f_!}\circ (\unit{f_*}{f^{\times}})_{f^*f_!} \\
& = f^{\times}\nu (\counit{f_!}{f^*}_{f_!})\circ \chi_{f_*f^*f_!}\circ  f^*(\lambda_{f_!f^*f_!})\circ \unit{f_!}{f^*}_{f^*f_!} \\
& = \chi_{f_*}\circ f^*\nu^-\nu (\counit{f_!}{f^*}_{f_!})\circ f^*(\lambda_{f_!f^*f_!})\circ \unit{f_!}{f^*}_{f^*f_!} \\
& = \chi_{f_*}\circ f^*(\lambda_{f_!})\circ f^*(\counit{f_!}{f^*}_{f_!})\circ \unit{f_!}{f^*}_{f^*f_!} \\
& = \chi_{f_*}\circ f^*(\lambda_{f_!})
\end{align*}
by naturality and the triangle identities.
\end{proof}

We leave it to the reader to state the dual of Theorem \ref{Isomorphism of monads adjoint pairs}.

We have the following uniqueness result for Nakayama functors.

\begin{Theorem}\label{Uniqueness Nakayama functor adjoint pairs}
Assume we have adjunctions $f_!\dashv f^*\dashv f_*$ where $f^*\colon \cA\to \cD$ is a faithful functor. Let $\nu_1$ and $\nu_2$ be Nakayama functors relative to $f_!\dashv f^*$ satisfying $\nu_1\circ f_! = f_*= \nu_2\circ f_!$.  The following holds:
\begin{enumerate}
\item\label{Uniqueness Nakayama functor adjoint pairs:1} There exists an equivalence $\Phi\colon \cA\to \cA$ satisfying $\Phi\circ f_!=f_!$ and a unique natural isomorphism $\xi\colon \nu_1\circ \Phi \to \nu_2$ satisfying $\xi_{f_!}=1_{f_*}$. Furthermore, if $\Psi$ is another such equivalence, then there exists a unique natural isomorphism $\zeta\colon \Phi\to \Psi$ satisfying $\zeta_{f_!}=1_{f_!}$. 
\item\label{Uniqueness Nakayama functor adjoint pairs:2} There exists an equivalence $\Phi'\colon \cA\to \cA$ satisfying $\Phi'\circ f_*=f_*$ and a unique natural isomorphism $\xi'\colon \Phi'\circ \nu_1 \to \nu_2$ satisfying $\xi_{f_!}=1_{f_*}$. Furthermore, if $\Psi'$ is another such equivalence, then there exists a unique natural isomorphism $\zeta\colon \Phi'\to \Psi'$ satisfying $\zeta_{f_!}=1_{f_*}$. 
\end{enumerate} 
\end{Theorem}

\begin{proof}
We only prove part $\ref{Uniqueness Nakayama functor adjoint pairs:1}$, part $\ref{Uniqueness Nakayama functor adjoint pairs:2}$ is proved similarly. Let $E_1:=\nu_1|_{\im f_!}\colon \im f_!\to \im f_*$ and $E_2:=\nu_2|_{\im f_!}\colon \im f_!\to \im f_*$, and let $H:=E_1^{-1}\circ E_2\colon \im f_!\to \im f_!$. Then obviously $H\circ f_!=f_!$. We define $\Phi:= \Coker \circ H\circ \mor_{f_!}$, where we choose the cokernels such that $\Phi|_{\im f_!} = H$. Since $\im f_!$ is generating, $\Phi$ is an equivalence on $\cA$. Furthermore, $\Phi\circ f_!=f_!$. Applying $\nu_1$ to the exact sequence 
\[
Hf_!f^*(\Ker \counit{f_!}{f^*}_{A})\to Hf_!f^*(A)\to \Phi(A)\to 0
\]
and using that $\nu_1|_{\im f_!}=E_1$ and $E_1H=E_2$, we get a commutative diagram with exact rows
\begin{equation*}
\begin{tikzpicture}[description/.style={fill=white,inner sep=2pt}]
\matrix(m) [matrix of math nodes,row sep=2.5em,column sep=2.5em,text height=1.5ex, text depth=0.25ex] 
{ E_1Hf_!f^*(\Ker \counit{f_!}{f^*}_{A}) & E_1Hf_!f^*(A) & \nu_1\Phi(A) & 0   \\
  E_2f_!f^*(\Ker \counit{f_!}{f^*}_{A}) & E_2f_!f^*(A)  & \nu_2(A) & 0  \\};
\path[->]
(m-1-1) edge node[auto] {$$} 	    												(m-1-2)
(m-1-2) edge node[auto] {$$} 	    												(m-1-3)
(m-1-3) edge node[auto] {$$} 	    												(m-1-4)
(m-2-1) edge node[auto] {$$} 	    												(m-2-2)
(m-2-2) edge node[auto] {$$} 	    												(m-2-3)
(m-2-3) edge node[auto] {$$} 	    												(m-2-4)

(m-1-1) edge node[auto] {$=$} 	    							   				    (m-2-1)
(m-1-2) edge node[auto] {$=$} 	    												(m-2-2)
(m-1-3) edge node[auto] {$\xi_{A}$} 	    											(m-2-3);	
\end{tikzpicture}
\end{equation*}
This defines $\xi_{A}$ for each $A\in \cA$. Obviously, $\xi$ satisfies $\xi_{f_!}=1_{f_*}$, and is the unique such natural isomorphism.

Now let $\Psi\colon \cA\to \cA$ be another such equivalence. Then the composite $\zeta'\colon \nu_1\circ \Phi \xrightarrow{\cong} \nu_2\xrightarrow{\cong} \nu_1 \circ \Psi$ gives a natural isomorphism satisfying $\zeta'_{f_!}=1_{f_*}$. Restricting to $\im f_!$ and using that $\nu_1|_{\im f_1}\colon \im f_!\to \im f_*$ is an isomorphism, we get that $\Phi|_{\im f_!}=\Psi|_{\im f_!}$. Hence, we have a commutative diagram 

\begin{equation*}
\begin{tikzpicture}[description/.style={fill=white,inner sep=2pt}]
\matrix(m) [matrix of math nodes,row sep=2.5em,column sep=5em,text height=1.5ex, text depth=0.25ex] 
{ f_!f^*(\Ker \counit{f_!}{f^*}_{A}) & f_!f^*(A) & \Phi(A) & 0   \\
  f_!f^*(\Ker \counit{f_!}{f^*}_{A}) & f_!f^*(A) & \Psi(A) & 0  \\};
\path[->]
(m-1-1) edge node[auto] {$\Phi(\mor_{f_!}(A))$} 	    							(m-1-2)
(m-1-2) edge node[auto] {$$} 	    												(m-1-3)
(m-1-3) edge node[auto] {$$} 	    												(m-1-4)
(m-2-1) edge node[auto] {$\Psi(\mor_{f_!}(A))$} 	    							(m-2-2)
(m-2-2) edge node[auto] {$$} 	    												(m-2-3)
(m-2-3) edge node[auto] {$$} 	    												(m-2-4)

(m-1-1) edge node[auto] {$=$} 	    							   				    (m-2-1)
(m-1-2) edge node[auto] {$=$} 	    												(m-2-2)
(m-1-3) edge node[auto] {$\zeta_{A}$} 	    											(m-2-3);	
\end{tikzpicture}
\end{equation*}
which defines $\zeta_{A}$ for each $A\in \cA$. Obviously $\zeta$ is the unique isomorphism satisfying $\zeta_{f_!}=1_{f_!}$. 
\end{proof}

Let $f^*\colon \cA\to \cB$ be a faithful functor between abelian categories with a left adjoint $f_!$, and assume $f_!\dashv f^*$ has a Nakayama functor as in Definition \ref{Nakayama functor for adjoint pair}. Then $f^*$ is faithful and exact. Let $\sfT=(f^*\circ f_!,\unit{f_!}{f^*},f^*(\counit{f_!}{f^*}_{f_!}))$ be the induced monad on $\cB$, and let $\cB^{\sfT}$ be the Eilenberg-Moore category of $\sfT$, see Definition \ref{Eilenberg-Moore category}. By Beck's monadicity theorem we have an equivalence of categories $J\colon\cA\xrightarrow{\cong} \cB^{\sfT}$ such that $U^{\sfT}\circ J$ is isomorphic to $f^*$, see Theorem \ref{Weak Beck's monadicity}. It is therefore natural to ask for necessary and sufficient conditions on a monad $\sfT$ on an abelian category $\cB$ so that the induced adjoint pair $F^{\sfT}\dashv U^{\sfT}$ on $\cB^{\sfT}$ has a Nakayama functor.

\begin{Theorem}\label{Admissible adjunction theorem}
Let $\sfT$ be a monad on an abelian category $\cB$, let $\cB^{\sfT}$ be the Eilenberg-Moore category of $\sfT$, and let $U^{\sfT}\colon \cB^{\sfT} \to \cB$ be the forgetful functor with left adjoint $F^{\sfT}\colon \cB\to \cB^{\sfT}$. Then the adjoint pair $F^{\sfT}\dashv U^{\sfT}$ has a Nakayama functor if and only if there exists a comonad $\sfS$ and an ambidextrous adjunction $\sfS\dashv \sfT\dashv \sfS$. 
\end{Theorem}

\begin{proof}
We set $f^*:=U^{\sfT}$ and $f_!:=F^{\sfT}$. If $f_!\dashv f^*$ has a Nakayama functor, then by Lemma \ref{Basic properties of Nakayama functor adjoint pairs} part $\ref{on adjunctions}$  we have adjunctions $f^?\dashv f_!\dashv f^*\dashv f_*$. Let $\sfS_1=(f^?\circ f_!,\counit{f^?}{f_!}, f^?(\unit{f^?}{f_!}_{f_!}))$ and $\sfS_2=(f^*\circ f_*,\counit{f^*}{f_*}, f^*(\unit{f^*}{f_*}_{f_*}))$ be the induced comonads on $\cB$.  It follows from \cite[Theorem 2.15]{Lau06} that $\sfS_1\dashv \sfT\dashv \sfS_2$. Finally, there exists an isomorphism $\sfS_1\cong \sfS_2$ of comonads by the dual of Theorem \ref{Isomorphism of monads adjoint pairs} part $\ref{Isomorphism of monads adjoint pairs:1}$. This shows that we have an ambidextrous adjunction.

Now assume there exists an ambidextrous adjunction $\sfS\dashv \sfT\dashv \sfS$ for some comonad $\sfS$ on $\cB$. It follows from \cite[Proposition 5.3]{EM65} that $\cB^{\sfT}$ is abelian. Furthermore, since $\sfT\dashv \sfS$, we have an isomorphism $J\colon \cB^{\sfS}\xrightarrow{\cong} \cB^{\sfT}$ of the Eilenberg-Moore categories which commute with the forgetful functors, see \cite[Theorem 2.14]{Lau06}. Hence, if we put $f_*:=J\circ F^{\sfS}$, we get an adjoint triple $f_!\dashv f^*\dashv f_*$ such that  $\sfS = (f^*\circ f_*,\counit{f^*}{f_*},f^*(\unit{f^*}{f_*}_{f_*}))$. Now by the dual of \cite[Lemma 4.3.3]{Bor94a} we have that for every $A\in \cB^{\sfT}$ there exists an exact sequence $0\to A\xrightarrow{\unit{f^*}{f_*}} f_*f^*(A) \xrightarrow{f_*f^*(\unit{f^*}{f_*}_A)-\unit{f^*}{f_*}_{f_*f^*(A)}} f_*f^*f_*f^*(A)$. Hence, in particular it follows that 
\[
f_! = \Ker (f_*f^*f_!\xrightarrow{f_*f^*(\unit{f^*}{f_*}_{f_!})-\unit{f^*}{f_*}_{f_*f^*f_!}} f_*f^*f_*f^*f_!)
\]

Since $\sfS\dashv \sfT$, we have that $f^*f_*\dashv f^*f_!$. It follows by Proposition \ref{prop:1.6} that there exists a natural transformation $s\colon f^*f_*f^*f_!f^*\to f^*f_*f^*$ such that $s$ and $f_*f^*(\unit{f^*}{f_*}_{f_!})-\unit{f^*}{f_*}_{f_*f^*f_!}$ are conjugate. If we let $f^?\colon \cB^{\sfT}\to \cB$ denote the cokernel of $s$, then it follows that $f^?\dashv f_!$ by Proposition \ref{prop:1.65}. Now consider the comonad $\sfS':= (f^?\circ f_!, \counit{f^?}{f_!},f^?(\unit{f^?}{f_!}_{f_!}))$. By \cite[Theorem 2.15]{Lau06} we have that $\sfS'\dashv \sfT$. But since the left adjoint of a functor is unique up to an isomorphism and the conjugate of a natural transformation is unique, it follows that the left adjoint of a monad is unique up to an isomorphism of comonads. Therefore, since $\sfS\dashv \sfT$, we get that $\sfS\cong \sfS'$. Hence, it follows that  
\[
(f^?\circ f_!, \counit{f^?}{f_!},f^?(\unit{f^?}{f_!}_{f_!}))\cong (f^*\circ f_*,\counit{f^*}{f_*},f^*(\unit{f^*}{f_*}_{f_*}))
\]
Finally, by the dual of Theorem \ref{Isomorphism of monads adjoint pairs} part $\ref{Isomorphism of monads adjoint pairs:2}$ it follows that $f_!\dashv f^*$ has a Nakayama functor.

\end{proof}

\section{Gorenstein homological algebra for Nakayama functors relative to endofunctors}\label{Gorenstein categories for comonads}

In this section we construct analogues of Gorenstein projective modules, Gorenstein injectives modules, and Iwanaga-Gorenstein algebras in the setting of Nakayama functors relative to endofunctors. Furthermore, we prove an analogue of Zak's result on the equality of the injective dimension for Iwanaga-Gorenstein algebras. We make the following assumption throughout the section.

\begin{Setting}\label{Nsetting:1}
Let $P\colon \cA\to \cA$ be a generating functor and let $\nu\colon \cA\to \cA$ be a Nakayama functor relative to $P$. We fix the notation $I:=\nu\circ P$, $T=P\circ \nu$, and $S:=I\circ \nu^-$.
\end{Setting}
Note that $T\dashv P\dashv I\dashv S$.

\subsection{Gorenstein \texorpdfstring{$P$}{}-projective and \texorpdfstring{$I$}{}-injective objects}

We call an object $A\in \cA$ for $P$-\emphbf{projective} or $I$-\emphbf{injective} if it is a summand of an object $P(A')$ or $I(A')$, respectively.

\begin{Definition}\label{Gorenstein objects for comonad with Nakayama functor paper}
Let $A\in \cA$ be an object.
\begin{enumerate}
\item\label{Gorenstein objects for comonad with Nakayama functor paper:1} An object $A\in \cA$ is \emphbf{Gorenstein} $P$-\emphbf{projective} if there exists an exact sequence 
\[
A_{\bullet}=\cdots \xrightarrow{s_{2}} A_{1}\xrightarrow{s_{1}} A_0\xrightarrow{s_0} A_{-1}\xrightarrow{s_{-1}} \cdots
\] 
with $A_i\in \cA$ being $P$-projective for all $i\in \bZ$, such that
\[
\cdots \xrightarrow{\nu(s_{2})} \nu(A_{1})\xrightarrow{\nu(s_{1})} \nu(A_0)\xrightarrow{\nu(s_0)} \nu(A_{-1})\xrightarrow{\nu(s_{-1})} \cdots
\] 
is exact, and with $Z_0(A_{\bullet})=\Ker s_0=A$. The subcategory of $\cA$ consisting of all Gorenstein $P$-projective objects is denoted by $\relGproj{P}{\cA}$
\item An object $A\in \cA$ is \emphbf{Gorenstein} $I$-\emphbf{injective} if there exists an exact sequence 
\[
A_{\bullet}=\cdots \xrightarrow{s_{2}} A_{1}\xrightarrow{s_{1}} A_0\xrightarrow{s_0} A_{-1}\xrightarrow{s_{-1}} \cdots
\] 
with $A_i\in \cA$ being $I$-injective for all $i\in \bZ$, such that
\[
\cdots \xrightarrow{\nu^-(s_{2})} \nu^-(A_{1})\xrightarrow{\nu^-(s_{1})} \nu^-(A_0)\xrightarrow{\nu^-(s_0)} \nu^-(A_{-1})\xrightarrow{\nu^-(s_{-1})} \cdots
\] 
is exact, and with $Z_0(A_{\bullet})=\Ker s_0=A$. The subcategory of $\cA$ consisting of all Gorenstein $I$-injective objects is denoted by $\relGinj{I}{\cA}$
\end{enumerate}
\end{Definition}

In Example \ref{One algebra} with $k$ a field and $P=f_!\circ f^*$, a module is Gorenstein $P$ projective or $I$-injective if and only if it is Gorenstein projective or injective, respectively. In Example \ref{Complexes} with $P=f_!\circ f^*$ all objects are Gorenstein $P$-projective and $I$-injective. In Example \ref{Morphism category} with $P=f_!\circ f^*$ a straightforward computation shows that a morphism is Gorenstein $P$-projective or $I$-injective if and only if it is a monomorphism or an epimorphism, respectively. 

The following result shows that $\relGproj{P}{\cA}$ and $\relGinj{I}{\cA}$ are equivalent categories.

\begin{Proposition}\label{1:prop:N1}
Let $A\in \cA$ be arbitrary. The following holds:
\begin{enumerate}
	\item\label{1:prop:N1,1} If $A\in \relGproj{P}{\cA}$, then $\nu(A)\in \relGinj{I}{\cA}$;
	\item\label{1:prop:N1,2} If $A\in \relGinj{I}{\cA}$, then $\nu^-(A)\in \relGproj{P}{\cA}$;
	\item\label{1:prop:N1,3} If $A\in \relGproj{P}{\cA}$, then $\lambda_A\colon A\to \nu^-\nu(A)$ is an isomorphism;
	\item\label{1:prop:N1,4} If $A\in \relGinj{I}{\cA}$, then $\sigma_A\colon \nu\nu^-(A)\to A$ is an isomorphism.
\end{enumerate}
In particular, the restriction $\nu|_{\relGproj{P}{\cA}}\colon \relGproj{P}{\cA}\to \relGinj{I}{\cA}$ is an equ\-ivalence with quasi-inverse $\nu^-|_{\relGinj{I}{\cA}}\colon \relGinj{I}{\cA}\to \relGproj{P}{\cA}$.
\end{Proposition}

\begin{proof}
Let $Q_{\bullet}=\cdots \xrightarrow{s_2} Q_{1}\xrightarrow{s_{1}} Q_0\xrightarrow{s_0} Q_{-1}\xrightarrow{s_{-1}} \cdots$ be an exact sequence as in the definition of Gorenstein $P$-projective. Applying $\nu$ gives an exact sequence
\[
\nu(Q_{\bullet})=\cdots \xrightarrow{\nu(s_{2})} \nu(Q_{1})\xrightarrow{\nu(s_{1})} \nu(Q_0)\xrightarrow{\nu(s_0)} \nu(Q_{-1})\xrightarrow{\nu(s_{-1})} \cdots
\] 
Since the components of $Q_{\bullet}$ are $P$-projective, we have an isomorphism $\nu^-\nu (Q_{\bullet})\cong Q_{\bullet}$, and therefore $\nu^-\nu (Q_{\bullet})$ is exact. Also, the complex $\nu(Q_{\bullet})$ has $I$-injective components since $\nu$ sends $P$-projective objects to $I$-injective objects. Hence, if $Z_0(Q_{\bullet})=A$, then $Z_0(\nu(Q_{\bullet}))=\nu(A)\in \relGinj{I}{\cA}$. This shows $\ref{1:prop:N1,1}$. Now consider the exact sequence $0\to \nu(A)\to\nu(Q_0)\xrightarrow{\nu(s_0)} \nu(Q_{-1})$. Applying $\nu^-$ gives a commutative diagram
\[
\begin{tikzpicture}[description/.style={fill=white,inner sep=2pt}]
\matrix(m) [matrix of math nodes,row sep=2.5em,column sep=5.0em,text height=1.5ex, text depth=0.25ex] 
{0 &  A & Q_0 & Q_{-1}  \\
 0 &  \nu^-\nu(A) &  \nu^-\nu(Q_0)   & \nu^-\nu(Q_{-1})      \\};
\path[->]
(m-1-1) edge node[auto] {$$} 	    													   	 			  (m-1-2)
(m-1-2) edge node[auto] {$$} 	    																			(m-1-3)
(m-1-3) edge node[auto] {$s_0$} 	    													 		 	  (m-1-4)
(m-2-1) edge node[auto] {$$} 	    													      			(m-2-2)
(m-2-2) edge node[auto] {$$} 	    																			(m-2-3)
(m-2-3) edge node[auto] {$\nu^-\nu(s_0)$} 	    									(m-2-4)
(m-1-2) edge node[auto] {$\lambda_A$} 	    								   					(m-2-2)
(m-1-3) edge node[auto] {$\lambda_{Q_0}$} 	    									 			(m-2-3)
(m-1-4) edge node[auto] {$\lambda_{Q_{-1}}$} 	    								 				(m-2-4);	
\end{tikzpicture}
\]
where the lower row is exact since $\nu^-$ is left exact. Since $\lambda_{Q_0}$ and $\lambda_{Q_{-1}}$ are isomorphisms, it follows that $\lambda_A$ is an isomorphism. This proves part $\ref{1:prop:N1,3}$. Part $\ref{1:prop:N1,2}$ and $\ref{1:prop:N1,4}$ are proved dually.
\end{proof}

In the following, note that if $A\in \relGproj{P}{\cA}$, then $L_i\nu(A)=0$ for all $i>0$. Dually, if $A\in \relGinj{I}{\cA}$, then $R^i\nu^-(A)=0$ for all $i>0$.

\begin{Proposition}\label{Gorenstein objects for comonad with Nakayama functor description}
Let $A\in \cA$. Then $A\in \relGproj{P}{\cA}$ if and only if the following holds:
\begin{enumerate}
\item\label{Gorenstein objects for comonad with Nakayama functor description:1} $L_i\nu(A)=0$ for all $i>0$;
\item\label{Gorenstein objects for comonad with Nakayama functor description:2} $R^i\nu^-(\nu(A))=0$ for all $i>0$;
\item\label{Gorenstein objects for comonad with Nakayama functor description:3} The unit $\lambda_A\colon A\to \nu^-\nu(A)$ is an isomorphism.
\end{enumerate}
\end{Proposition}

\begin{proof}
Assume $A\in \relGproj{P}{\cA}$, and let $A_{\bullet}$ be the complex in the definition of Gorenstein $P$-projective. Consider the exact sequence 
\[
0\to \Ker s_1\to A_1\to A\to 0
\]
Applying $\nu$ to this gives a long exact sequence
\begin{align*}
\cdots \to L_1\nu(A_1)\to L_1\nu(A) \to \nu(\Ker s_1)\to \nu(A_1)\to \nu(A)\to 0
\end{align*} 
But by assumption the sequence $0\to \nu(\Ker s_1)\to \nu(A_1)\to \nu(A)\to 0 $ is exact and $L_i\nu(A_1)=0$ for all $i>0$. Hence, it follows that $L_1\nu(A)=0$. Since $\Ker s_1$ is also Gorenstein $P$-projective and $L_i\nu(\Ker s_1)\cong L_{i+1}\nu(A)$ for all $i>0$, it follows by induction that $L_i\nu(A)=0$ for all $i>0$. Since $\nu(A)\in \relGinj{I}{\cA}$, a dual argument shows that $R^i\nu^-(\nu(A))=0$ for all $i>0$. Finally, $\lambda_A\colon A\to \nu^-\nu(A)$ is an isomorphism by Proposition \ref{1:prop:N1}.

Assume $A$ satisfies $\ref{Gorenstein objects for comonad with Nakayama functor description:1}$, $\ref{Gorenstein objects for comonad with Nakayama functor description:2}$, and $\ref{Gorenstein objects for comonad with Nakayama functor description:3}$. Choose an exact sequence $0\to \nu(A) \to J_{-1}\to J_{-2}\to \cdots$ with $J_i$ being $I$-injective for all $i$. Applying $\nu^-$ and using that $A\cong \nu^-\circ \nu(A)$ and $R^i\nu^-(\nu(A))=0$ for all $i$, we get an exact sequence 
\[
0\to A\to \nu^-(J_{-1})\to \nu^-(J_{-2})\to \cdots .
\] 
which is still exact when one applies $\nu$. Now choose an exact sequence $\cdots \to Q_{1}\to Q_0\to A\to 0$ with $Q_i$ being $P$-projective for all $i$, and consider the exact sequence 
\[
Q_{\bullet} = \cdots \to Q_{1}\to Q_0\to \nu^-(J_{-1})\to \nu^-(J_{-2})\to \cdots 
\] 
obtained by gluing the two sequence together at $A$. Since $L_i\nu(A)=0$ for all $i>0$, this sequence is still exact when one applies $\nu$. Furthermore, it has $P$-projective components and satisfy $Z_0(Q_{\bullet})=A$. This shows that $A\in \relGproj{P}{\cA}$. 
\end{proof}

We have the following dual of Proposition \ref{Gorenstein objects for comonad with Nakayama functor description}.

\begin{Proposition}\label{dual Gorenstein objects for comonad with Nakayama functor}
Let $A\in \cA$. Then $A\in \relGinj{I}{\cA}$ if and only if the following holds:
\begin{enumerate}
\item\label{dual Gorenstein objects for comonad with Nakayama functor:1} $R^i\nu^-(A)=0$ for all $i>0$;
\item\label{dual Gorenstein objects for comonad with Nakayama functor:2} $L_i\nu(\nu^-(A))=0$ for all $i>0$;
\item\label{dual Gorenstein objects for comonad with Nakayama functor:3} The counit $\sigma_A\colon \nu\nu^-(A)\to A$ is an isomorphism.
\end{enumerate}
\end{Proposition}

\begin{Proposition}\label{resolving and coresolving}
The following holds:
\begin{enumerate}
\item $\relGproj{P}{\cA}$ is a resolving subcategory of $\cA$;
\item $\relGinj{I}{\cA}$ is a coresolving subcategory of $\cA$.
\end{enumerate}
\end{Proposition}

\begin{proof}
This follows from the description of $\relGproj{P}{\cA}$ and $\relGinj{I}{\cA}$ in Proposition \ref{Gorenstein objects for comonad with Nakayama functor description} and Proposition \ref{dual Gorenstein objects for comonad with Nakayama functor}.
\end{proof} 

Hence, we can define the resolving dimension $\dim_{\relGproj{P}{\cA}}(-)$ and the coresolving dimension $\dim_{\relGinj{I}{\cA}}(-)$ as described in Subsection \ref{Properties of subcategories}.

We end this subsection with the following result

\begin{Lemma}\label{Cokernels of Gorenstein relative projectives}
The following holds:
\begin{enumerate}
\item\label{Cokernels of Gorenstein relative projectives:1} If $A_0\xrightarrow{i} A_1\to A_2\to 0$ is an exact sequence where $A_0,A_1\in \relGproj{P}{\cA}$ and such that $\nu(i)$ is a monomorphism, then $i$ is a monomorphism and $A_2\in \relGproj{P}{\cA}$.
\item\label{Cokernels of Gorenstein relative projectives:2} If $0\to A_0\to A_1\xrightarrow{p} A_2$ is an exact sequence with $A_1,A_2\in \relGinj{I}{\cA}$ and such that $\nu^-(p)$ is an epimorphism, then $p$ is an epimorphism and $A_0\in \relGinj{I}{\cA}$.
\end{enumerate}
\end{Lemma}

\begin{proof}
We prove part $\ref{Cokernels of Gorenstein relative projectives:1}$, part $\ref{Cokernels of Gorenstein relative projectives:2}$ is proved dually. Since $\nu(i)$ is a monomorphism, it follows that $\nu^-\nu(i)$ is a monomorphism. Since $\nu^-\nu(i)$ corresponds to $i$ under the isomorphisms $A_0\cong \nu^-\nu(A_0)$ and $A_1\cong \nu^-\nu(A_1)$, it follows that $i$ is a monomorphism. The fact that $A_2\in \relGproj{P}{\cA}$ follows from Proposition \ref{Gorenstein objects for comonad with Nakayama functor description}.
\end{proof}

\subsection{Iwanaga-Gorenstein functors}

\begin{Definition}\label{definition:N2}
Assume $P$ is a generating functor with Nakayama functor $\nu$ relative to $P$. Then $P$ is called \emphbf{Iwanaga-Gorenstein} if there exists an $m\geq 0$ such that $L_i\nu(A)=0$ and $R^i\nu^-(A)=0$ for all $A\in \cA$ and $i>m$.
\end{Definition}

We have a simpler description of $\relGproj{P}{\cA}$ and $\relGinj{I}{\cA}$ when $P$ is Iwanaga-Gorenstein.

\begin{Theorem}\label{thm:N1}
Assume $P$ is Iwanaga-Gorenstein. The following holds:
\begin{enumerate}
	\item\label{thm:N1,1} $A \in \relGproj{P}{\cA}$ if and only if $L_i\nu(A)=0$ for all $i>0$;
	\item\label{thm:N1,2} $A \in \relGinj{I}{\cA}$ if and only if $R^i\nu^-(A)=0$ for all $i>0$.
\end{enumerate}
\end{Theorem}

\begin{proof}
If $A \in \relGproj{P}{\cA}$, then $L_i\nu(A)=0$ for all $i>0$ by Proposition \ref{Gorenstein objects for comonad with Nakayama functor description}. For the converse, fix a number $m$ such that $L_i\nu=0$ and $R^i\nu^-=0$ for all $i>m$, and note that by Proposition \ref{Gorenstein objects for comonad with Nakayama functor description}  we only need to show that $R^i\nu^-(\nu(A))=0$ and $\lambda_A\colon A\to \nu^-\nu(A)$ is an isomorphism. To this end, choose an exact sequence
\[
\cdots \xrightarrow{s_{3}}Q_{2}\xrightarrow{s_{2}}Q_{1}\xrightarrow{s_{1}}A\to 0
\]
with $Q_{i}$ being $P$-projective. Applying $\nu$ gives an exact sequence
\[
\cdots \xrightarrow{\nu(s_{3})}\nu(Q_{2})\xrightarrow{\nu(s_{2})}\nu(Q_{1})\xrightarrow{\nu(s_{1})}\nu(A)\to 0
\] 
since $L_i\nu(A)=0$ for all $i>0$. Also, since $\nu(Q_i)$ is $I$-injective, it follows by dimension shifting that $R^i\nu^-(\nu(A))\cong R^{i+m}\nu^-(\Ker \nu(s_m))=0$ and $R^i\nu^-(\Ker \nu(s_j))\cong R^{i+m}\nu^-(\Ker \nu(s_{j+m}))=0$ for all $i>0$ and $j>0$. Therefore, we have a commutative diagram 
\[
\begin{tikzpicture}[description/.style={fill=white,inner sep=2pt}]
\matrix(m) [matrix of math nodes,row sep=3.5em,column sep=3.5em,text height=1.5ex, text depth=0.25ex] 
{ \cdots & Q_{2} & Q_{1} & A & 0 \\
  \cdots & \nu^-\nu(Q_{2}) & \nu^-\nu(Q_{1}) & \nu^-\nu(A) & 0\\};
\path[->]
(m-1-1) edge node[auto] {$s_{3}$} 	    													   	 	  (m-1-2)
(m-1-2) edge node[auto] {$s_{2}$} 	    																	(m-1-3)
(m-1-3) edge node[auto] {$s_{1}$} 	    													 		  	(m-1-4)
(m-1-4) edge node[auto] {$$} 	    													 	    				(m-1-5)
(m-2-1) edge node[auto] {$\nu^-\nu(s_{3})$} 	    									(m-2-2)
(m-2-2) edge node[auto] {$\nu^-\nu(s_{2})$} 	    									(m-2-3)
(m-2-3) edge node[auto] {$\nu^-\nu(s_{1})$} 	    									(m-2-4)
(m-2-4) edge node[auto] {$$} 	    														    				(m-2-5)
(m-1-2) edge node[auto] {$\lambda_{Q_{2}}$} 	    								   					  (m-2-2)
(m-1-3) edge node[auto] {$\lambda_{Q_{1}}$} 	    									  (m-2-3)
(m-1-4) edge node[auto] {$\lambda_{A}$} 	    								 						  	(m-2-4);	
\end{tikzpicture}
\]
where the rows are exact. Hence, the morphism $\lambda_{A}\colon A\to \nu^-\nu(A)$ is an isomorphism. This proves part $\ref{thm:N1,1}$. Part $\ref{thm:N1,2}$ is proved dually.
\end{proof}

It follows from Theorem \ref{thm:N1} that if $P$ is Iwanaga-Gorenstein, then 
\[
\dim_{\relGproj{P}{\cA}}(\cA)<\infty \quad \text{and} \quad  \dim_{\relGinj{I}{\cA}}(\cA)<\infty.
\]

Our goal is to prove that when $P$ is Iwanaga-Gorenstein the following numbers are equal:
\begin{itemize}
	\item[1)] $\dim_{\relGproj{P}{\cA}}(\cA)$;
	\item[2)] $\dim_{\relGinj{I}{\cA}}(\cA)$;
	\item[3)] The smallest integer $s\geq 0$ such that $L_i\nu(A)=0$ for all $i>s$ and $A\in \cA$;
	\item[4)] The smallest integer $t\geq 0$ such that $R^i\nu^-(A)=0$ for all $i>t$ and $A\in \cA$.
\end{itemize}
We also show that if the numbers in 1) or 2) are finite, then $P$ is Iwanaga-Gorenstein. This generalizes the finite-dimensional case, see \cite[Proposition 3.1 part b)]{AR91a} and \cite[Proposition 4.2]{AR91a}. It is also analogous to other results in Gorenstein homological algebra, see for example \cite[Theorem 2.28]{EEG08}.

In order to prove this we need some preparation. 

\begin{Lemma}\label{lemma:N3}
Let $A\in \cA$. The following holds:
\begin{enumerate}
	\item\label{lemma:N3,1}  $A\cong \nu(A')$ for some $A'\in \cA$ if and only if there exists an exact sequence $J_{0}\to J_1\to A\to 0$ with $J_0$ and $J_1$ being $I$-injective;
	\item\label{lemma:N3,2}  $A\cong \nu^-(A')$ for some $A'\in \cA$ if and only if there exists an exact sequence $0\to A\to Q_{0}\to Q_1$ with $Q_0$ and $Q_1$ being $P$-projective.
\end{enumerate}
\end{Lemma}

\begin{proof}
For any object $A'\in \cA$ choose an exact sequence $Q_0\to Q_1\to A'\to 0$ with $Q_0$ and $Q_1$ being $P$-projective. By applying $\nu$ and using that it is right exact and sends $P$-projective objects to $I$-injective objects, we get one direction of part $\ref{lemma:N3,1}$. For the converse, assume we have an exact sequence $J_{0}\xrightarrow{s} J_1\to A\to 0$ with $J_0$ and $J_1$ being $I$-injective. Since $\sigma_{J_i}\colon \nu\nu^-(J_i)\to J_i$ is an isomorphism, it follows that 
\[
A=\Coker s\cong \Coker \nu\nu^-(s)\cong \nu(\Coker \nu^-(s)). 
\]
This proves part $\ref{lemma:N3,1}$. Part $\ref{lemma:N3,2}$ is proved dually.
\end{proof}

Let $\Omega:= \Ker \counit{P}{I}\colon \cA\to \cA$ and $\Sigma := \Coker \unit{T}{P}\colon \cA\to \cA$. We then have exact sequences 
\begin{align*}
& 0\to \Omega(A)\to PI(A)\xrightarrow{\counit{P}{I}_A}A\to 0 \\
&  A\xrightarrow{\unit{T}{P}_A}PT(A)\to \Sigma(A)\to 0
\end{align*} By Proposition \ref{Conjugate units and counits} and Proposition \ref{prop:1.65} we have that $\Sigma \dashv \Omega$. Hence, for $r\geq 0$ the functor $\Omega^r\circ \nu^-$ is right adjoint to $\nu\circ \Sigma^r$.
 
\begin{Lemma}\label{lemma:N4}
Let $A\in \cA$ and $r\geq 0$ be an integer. The following holds:
\begin{enumerate}
	\item\label{lemma:N4,1} If $A\cong \nu\Sigma^r(A')$ for an object $A'\in \cA$, then there exists an exact sequence
	\[
	J_{0}\to J_1\to \cdots \to J_{r+1}\to A\to 0
	\]
	with $J_i$ being $I$-injective;
	\item\label{lemma:N4,2} If $A\in \relGproj{P}{\cA}$, then $\nu\Sigma^r(A)\in \relGinj{I}{\cA}$;
	\item\label{lemma:N4,3} If $A\in \relGinj{I}{\cA}$, then $\Omega^r\nu^-(A)\in \relGproj{P}{\cA}$.
	
\end{enumerate}
\end{Lemma}

\begin{proof}
We prove $\ref{lemma:N4,1}$. Consider the sequence
\begin{multline}\label{equation:N1}
A'\xrightarrow{\unit{T}{P}_{A'}} PT(A') \xrightarrow{s_0} PT\Sigma(A')\xrightarrow{s_1} \cdots \xrightarrow{s_{r-1}} PT\Sigma^{r-1}(A') \\
\xrightarrow{p_{r-1}} \Sigma^r(A')\to 0
\end{multline}
where $p_i$ is the canonical projection and $s_i$ is the composite 
\[
PT\Sigma^i(A')  
\xrightarrow{p_i} \Sigma^{i+1}(A')\xrightarrow{\unit{T}{P}_{\Sigma^{i+1}(A')}}PT\Sigma^{i+1}(A').
\]
Note that $T(\unit{T}{P})\colon T\to T\circ P\circ T$ is a split monomorphism by the triangle identities. Since $T=P\circ \nu$ and $P$ is faithful by Lemma \ref{lemma:N1}, it follows that $\nu(\unit{T}{P})\colon \nu \to \nu\circ P\circ T$ is a monomorphism. Hence, applying $\nu$ to \eqref{equation:N1} gives an exact sequence
\begin{multline*}
0\to \nu(A')\xrightarrow{\nu(\unit{T}{P}_{A'})} IT(A') \xrightarrow{\nu(s_0)} IT\Sigma(A')\xrightarrow{\nu(s_1)} \cdots \xrightarrow{\nu(s_{r-1})} IT\Sigma^{r-1}(A') \\
\xrightarrow{} A\to 0.
\end{multline*}
This proves part $\ref{lemma:N4,1}$. Finally, note that if $A\in \relGproj{P}{\cA}$, then $\Omega(A)\in \relGproj{P}{\cA}$ since $\relGproj{P}{\cA}$ is resolving, and $\Sigma(A)\in \relGproj{P}{\cA}$ by Lemma \ref{Cokernels of Gorenstein relative projectives}. Since $\nu(A)\in \relGinj{I}{\cA}$ and $\nu^-(A')\in \relGproj{P}{\cA}$ for $A\in \relGproj{P}{\cA}$ and $A'\in \relGinj{I}{\cA}$ by Proposition \ref{1:prop:N1}, we get part $\ref{lemma:N4,2}$ and $\ref{lemma:N4,3}$.
\end{proof}

We now prove the main result of this subsection.

\begin{Theorem}\label{thm:N2}
The following are equivalent:
\begin{enumerate}
	\item\label{thm:N2,1} $P$ is Iwanaga-Gorenstein;
	\item\label{thm:N2,2} $\dim_{\relGproj{P}{\cA}}(\cA)< \infty$;
	\item\label{thm:N2,3} $\dim_{\relGinj{I}{\cA}}(\cA)< \infty$.
\end{enumerate}
Moreover, if this holds, then the following numbers coincide:
\begin{enumerate}[label=(\alph*)]
\item\label{thm:N2,1,1} $\dim_{\relGproj{P}{\cA}}(\cA)$;
\item\label{thm:N2,1,2} $\dim_{\relGinj{I}{\cA}}(\cA)$;
\item\label{thm:N2,1,3} The smallest integer $s$ such that $L_i\nu(A)=0$ for all $i>s$ and $A\in \cA$;
\item\label{thm:N2,1,4} The smallest integer $t$ such that $R^i\nu^-(A)=0$ for all $i>t$ and $A\in \cA$.
\end{enumerate}
\end{Theorem}

If this common number is $n$, we say that $P$ is $n$-\emphbf{Gorenstein}.

\begin{proof}
We prove that $\dim_{\relGproj{P}{\cA}}(\cA)=\dim_{\relGinj{I}{\cA}}(\cA)$ by showing that
\[
\dim_{\relGproj{P}{\cA}}(\cA)\leq n \quad \text{if and only if} \quad \dim_{\relGinj{I}{\cA}}(\cA)\leq n
\]
for any number $n\geq 0$. This together with Theorem \ref{thm:N1} shows the equivalence of statement $\ref{thm:N2,1}$, $\ref{thm:N2,2}$, and $\ref{thm:N2,3}$, as well as the equality of the numbers in $\ref{thm:N2,1,1}$, $\ref{thm:N2,1,2}$, $\ref{thm:N2,1,3}$, and $\ref{thm:N2,1,4}$. 

First assume $n\geq 2$ and $\dim_{\relGinj{I}{\cA}}(\cA)\leq n$. Let $A\in \cA$ and consider the exact sequence
\begin{multline*}
0\to \Omega^n(A)\xrightarrow{i_n} PI\Omega^{n-1}(A)\xrightarrow{s_{n-1}} \cdots \\
\cdots \xrightarrow{s_3} PI\Omega^2(A) \xrightarrow{s_2} PI\Omega(A) \xrightarrow{s_1} PI(A)\xrightarrow{\counit{P}{I}_A} A\to 0
\end{multline*}
where $s_j$ is the composition
\[
PI\Omega^{j}(A) \xrightarrow{\counit{P}{I}_{\Omega^{j}(A)}} \Omega^{j}(A) \xrightarrow{i_j} PI\Omega^{j-1}(A)
\]
and $i_j$ is the inclusion. By Lemma \ref{lemma:N3} part $\ref{lemma:N3,2}$ there exists an object $A'\in \cA$ such that $\Omega^{2}(A)\cong \nu^-(A')$. This implies that 
\[
\Omega^n(A)\cong \Omega^{n-2}\nu^-(A').
\]
By Lemma \ref{lemma:N4} part $\ref{lemma:N4,1}$ and our assumption we know that $\nu \Sigma^{n-2}(A'')\in \relGinj{I}{\cA}$ for all $A''\in \cA$. Hence, by Lemma \ref{lemma:N4} part $\ref{lemma:N4,3}$ it follows that 
\[
\Omega^{n-2}\nu^-\nu\Sigma^{n-2}\Omega^{n-2}\nu^-(A')\in \relGproj{P}{\cA}.
\] 
By the triangle identities for $\nu \circ \Sigma^{n-2}\dashv \Omega^{n-2}\circ \nu^-$, we get that $\Omega^n(A)\cong \Omega^{n-2}\nu^-(A')$ is a direct summand of $\Omega^{n-2}\nu^-\nu\Sigma^{n-2}\Omega^{n-2}\nu^-(A')$. It follows that 
\[
\Omega^n(A)\in\relGproj{P}{\cA}
\]
since $\relGproj{P}{\cA}$ is closed under direct summands. This shows that 
\[
\dim_{\relGproj{P}{\cA}}(\cA)\leq n.
\]

Now assume $\dim_{\relGinj{I}{\cA}}(\cA)\leq 1$. By the argument above we know that $\dim_{\relGproj{P}{\cA}}(\cA)\leq 2$. Let $A\in \cA$ by arbitrary, and choose an exact sequence
\[
0\to \Ker s\xrightarrow{i} Q_{0}\xrightarrow{s} Q_1\xrightarrow{p} A\to 0
\]
with $Q_0,Q_1$ being $P$-projective. Since $\dim_{\relGproj{P}{\cA}}(\cA)\leq 2$, we get that $\Ker s\in \relGproj{P}{\cA}$. Consider the exact sequence $0\to \Ker s \xrightarrow{i}Q_0\xrightarrow{q} \im s\to 0$. Applying $\nu$ to this gives an exact sequence
\[
\nu(\Ker s) \xrightarrow{\nu(i)}\nu(Q_0)\xrightarrow{\nu(q)} \nu(\im s)\to 0.
\] 
Hence, we have an epimorphism $\nu(\Ker s)\xrightarrow{p'} \Ker \nu(q)\to 0$. Since $\nu(\Ker s)\in \relGinj{I}{\cA}$ and $\dim_{\relGinj{I}{\cA}}(\cA)\leq 1$, it follows that $\Ker \nu(q)\in \relGinj{I}{\cA}$. Furthermore, we have a commutative diagram 
\[
\begin{tikzpicture}[description/.style={fill=white,inner sep=2pt}]
\matrix(m) [matrix of math nodes,row sep=2.5em,column sep=5.0em,text height=1.5ex, text depth=0.25ex] 
{ \Ker s & Q_0 & Q_1 \\
  \nu^-\nu(\Ker s) &  \nu^-\nu(Q_0)   & \nu^-\nu(Q_1)  \\};
\path[->]
(m-1-1) edge node[auto] {$i$} 	    													   	  				(m-1-2)
(m-1-2) edge node[auto] {$s$} 	    																				(m-1-3)
(m-2-1) edge node[auto] {$\nu^-\nu(i)$} 	    													    (m-2-2)
(m-2-2) edge node[auto] {$\nu^-\nu(s)$} 	    															(m-2-3)
(m-1-1) edge node[auto] {$\lambda_{\Ker s}$} 	    								   				(m-2-1)
(m-1-2) edge node[auto] {$\lambda_{Q_0}$} 	    									  				(m-2-2)
(m-1-3) edge node[auto] {$\lambda_{Q_1}$} 	    								 						(m-2-3);	
\end{tikzpicture}
\] 
The vertical morphisms are isomorphisms by Proposition \ref{1:prop:N1} part $\ref{1:prop:N1,3}$. Hence, the morphism $\nu^-\nu(\Ker s)\xrightarrow{\nu^-\nu(i)}\nu^-\nu(Q_0)$ is the kernel of $\nu^-\nu(s)$. In particular, it is a monomorphism. On the other hand, $\nu^-\nu(i)$ is also equal to the composition 
\[
\nu^-\nu(\Ker s)\xrightarrow{\nu^-(p')}\nu^-(\Ker \nu(q))\xrightarrow{\nu^-(j)}\nu^-\nu(Q_0)
\]
where $j\colon \Ker \nu(q)\to \nu(Q_0)$ is the inclusion. Since $\nu^-\nu(s)\circ \nu^-(j) =0$ and $\nu^-(j)$ is a monomorphism, it follows that $\nu^-(p')$ is an isomorphism. Now consider the commutative diagram

\[
\begin{tikzpicture}[description/.style={fill=white,inner sep=2pt}]
\matrix(m) [matrix of math nodes,row sep=2.5em,column sep=5.0em,text height=1.5ex, text depth=0.25ex] 
{ \nu\nu^-\nu(\Ker s) & \nu\nu^-(\Ker \nu(q)) \\
  \nu(\Ker s) &  \Ker \nu(q)   \\};
\path[->]
(m-1-1) edge node[auto] {$\nu\nu^-(p')$} 	    													   	  	(m-1-2)
(m-2-1) edge node[auto] {$p'$} 	    													    							(m-2-2)
   									
(m-1-1) edge node[auto] {$\sigma_{\nu(\Ker s)}$} 	    								   				(m-2-1)
(m-1-2) edge node[auto] {$\sigma_{\Ker \nu(q)}$} 	    									  			(m-2-2);	
\end{tikzpicture}
\] 
Since the vertical maps and the upper horizontal map are isomorphisms, it follows that $p'$ is an isomorphism. Hence, the exact sequence
$0\to \Ker s \xrightarrow{i}Q_0\xrightarrow{q} \im s\to 0$ is still exact after one applies $\nu$. By Lemma \ref{Cokernels of Gorenstein relative projectives} it follows that $\im s\in \relGproj{P}{\cA}$. This implies that $\dim_{\relGproj{P}{\cA}}(A)\leq 1$, and since $A$ was arbitrary we get that $\dim_{\relGproj{P}{\cA}}(\cA)\leq 1$.

Finally, we consider the case when $\dim_{\relGinj{I}{\cA}}(\cA)=0$. This implies that $\nu^-$ is exact. Also, $\dim_{\relGproj{P}{\cA}}(\cA)\leq 1$ by the argument above. Let $A\in \cA$ be arbitrary, and choose a right exact sequence
\[
Q_0 \xrightarrow{s} Q_1 \xrightarrow{p} A\to 0
\]
with $Q_0,Q_1$ being $P$-projective. Since $\nu^-$ is exact and $\nu$ is right exact, the sequence $\nu^-\nu(Q_0) \xrightarrow{\nu^-\nu(s)} \nu^-\nu(Q_1) \xrightarrow{\nu^-\nu(p)} \nu^-\nu(A)\to 0$ is exact. Hence we have a commutative diagram 
\[
\begin{tikzpicture}[description/.style={fill=white,inner sep=2pt}]
\matrix(m) [matrix of math nodes,row sep=2.5em,column sep=5.0em,text height=1.5ex, text depth=0.25ex] 
{ Q_0 & Q_1 & A & 0 \\
  \nu^-\nu(Q_0) &  \nu^-\nu(Q_1)   & \nu^-\nu(A) & 0  \\};
\path[->]
(m-1-1) edge node[auto] {$s$} 	    													   	  				(m-1-2)
(m-1-2) edge node[auto] {$p$} 	    																				(m-1-3)
(m-1-3) edge node[auto] {$$} 	    																				  (m-1-4)
(m-2-1) edge node[auto] {$\nu^-\nu(s)$} 	    													    (m-2-2)
(m-2-2) edge node[auto] {$\nu^-\nu(p)$} 	    															(m-2-3)
(m-2-3) edge node[auto] {$$} 	    																				  (m-2-4)
(m-1-1) edge node[auto] {$\lambda_{Q_0}$} 	    								   				  (m-2-1)
(m-1-2) edge node[auto] {$\lambda_{Q_1}$} 	    									  				(m-2-2)
(m-1-3) edge node[auto] {$\lambda_{A}$} 	    								 						  (m-2-3);	
\end{tikzpicture}
\] 
with right exact rows. Since $\lambda_{Q_0}$ and $\lambda_{Q_1}$ are isomorphisms, it follows that $\lambda_A$ is an isomorphism. Since $\dim_{\relGproj{P}{\cA}}(\cA)\leq 1$, it follows by Lemma \ref{lemma:N3} part $\ref{lemma:N3,2}$ that $\nu(A)\in \relGinj{I}{\cA}$, and hence $A\cong \nu^-\nu(A)\in \relGproj{P}{\cA}$. Since $A$ was arbitrary, we get that $\dim_{\relGproj{P}{\cA}}(\cA)=0$. 

The dual of the above argument shows that if $\dim_{\relGproj{P}{\cA}}(\cA)\leq n$, then $\dim_{\relGinj{I}{\cA}}(\cA)\leq n$. Hence, it follows that 
\[
\dim_{\relGproj{P}{\cA}}(\cA)= \dim_{\relGinj{I}{\cA}}(\cA).
\]
and we are done.
\end{proof}

Note that in Example \ref{Complexes} the functor $P=f_!\circ f^*$ is $0$-Gorenstein, and in Example \ref{Morphism category} it is $1$-Gorenstein.

\section{Functor categories}\label{Functor categories section}

Let $k$ be a commutative ring, let $\cB$ be a $k$-linear abelian category, and let be $\cC$ a small, $k$-linear, locally bounded, and Hom-finite category. Our goal in this section is to show that there exists an adjoint pair $i_!\dashv i^*$ with Nakayama functor $\nu$ where 
\[
i^*\colon \cB^{\cC} \to \prod_{c\in \cC}\cB
\]
is the evaluation functor. 

\subsection{Reminder on the Hom and tensor functor}\label{Hom and tensor functor}

Let $k$ be a commutative ring. The \emphbf{tensor product} of two $k$-linear categories $\cD$ and $\cE$ is a $k$-linear category $\cD\otimes \cE$ with objects pairs $(D,E)$ with $D\in \cD$ and $E\in \cE$. The set morphisms between $(D,E)$ and $(D',E')$ is $\cD(D,D')\otimes_k \cE(E,E')$, and composition is given by $(h_1\otimes g_1)\circ (h_2\otimes g_2) = (h_1\circ h_2)\otimes (g_1\circ g_2)$. The identity at $(D,E)$ is $1_{D\otimes E} = 1_D\otimes 1_E$.

Let $\cC$ be a small $k$-linear category. Recall that the Yoneda lemma gives a fully faithful functor $h_{\cC}\colon \cC \to \Md\text{-}\cC$, where $\Md\text{-}\cC$ is the category of right $\cC$-modules. The image $h_{\cC}(c)=\cC(-,c)$ is a projective $\cC$-module. A right $\cC$-module $M$ is called \emphbf{finitely presented} if there exists an exact sequence  
\[
\oplus_{i=1}^m\cC(-,c_i) \to \oplus_{j=1}^n\cC(-,d_j)\to M\to 0 
\]
in $\Md \text{-}\cC$ for objects $c_i, d_j\in \cC$. The category of finitely presented right $\cC$-modules is denoted by $\md \text{-}\cC$. This is an additive category with cokernels. 

Let $\cB$ be a $k$-linear abelian category, and let $\cB^{\cC}$ denote the category of $k$-linear functors from $\cC$ to $\cB$. Since $\cB$ is finitely complete and cocomplete, it follows from chapter 3 in \cite{Kel05} that there exist functors
\begin{align*}
& -\otimes_{\cC}- \colon (\md\text{-} \cC) \otimes \cB^{\cC} \to \cB \\
& \Hom_{\cC}(-.-) \colon (\md\text{-} \cC\op)\op \otimes \cB^{\cC} \to \cB
\end{align*}
satisfying the following: 
\begin{enumerate}
\item\label{Tensor 1} The functor $-\otimes_{\cC} F\colon\md\text{-} \cC \to \cB$ is right exact for all $F\in \cB^{\cC}$.
\item\label{Tensor 2} There exists an isomorphism $\cC(-,c)\otimes_{\cC} F \cong F(c)$ natural in $c$ and $F$.
\item\label{Hom 1} The functor $\Hom_{\cC}(-,F)\colon (\md\text{-} \cC\op)\op \to \cB$ is left exact for all $F\in \cB^{\cC}$
\item\label{Hom 2} There exists an isomorphism $\Hom_{\cC}(\cC(c,-),F)\cong F(c)$ natural in $c$ and $F$.
\end{enumerate}
In \cite{Kel05} $M\otimes_{\cC}F$ and $\Hom_{\cC}(N,F)$ is called the \emphbf{colimit of $F$ indexed by $M$} and the \emphbf{limit of $F$ indexed by $N$}, respectively.  Note that \ref{Tensor 1} and \ref{Tensor 2} determines $-\otimes_{\cC}-$ uniquely, and \ref{Hom 1} and \ref{Hom 2} determines $\Hom_{\cC}(-,-)$ uniquely. If $\cC=k$ we get functors $-\otimes_{k}- \colon (\md\text{-} k) \otimes \cB \to \cB$ and $\Hom_{k}(-.-) \colon (\md\text{-} k)\op \otimes \cB \to \cB$, respectively.

Now let $\cC_1$ and $\cC_2$ be small $k$-linear categories. Assume $M\in \Md\text{-} (\cC_1\otimes \cC_2\op)$ satisfies $M(c_1,-)\in \md\text{-} \cC_2\op$ and $M(-,c_2)\in \md \text{-}\cC_1$ for all $c_1\in \cC_1$ and $c_2\in \cC_2$. We then have functors 
\[
M\otimes_{\cC_1}-\colon \cB^{\cC_1}\to \cB^{\cC_2} \quad \text{and} \quad \Hom_{\cC_2}(M,-)\colon \cB^{\cC_2}\to \cB^{\cC_1}
\]
and an isomorphism 
\begin{equation}\label{Adjunction}
\cB^{\cC_2}(M\otimes_{\cC_1} F, G)\cong \cB^{\cC_1}(F,\Hom_{\cC_2}(M,G)).
\end{equation}
natural in $F\in \cB^{\cC_1}$, $G\in \cB^{\cC_2}$ and $M$. It follows that $M\otimes_{\cC_1}-$ is left adjoint to $\Hom_{\cC_2}(M,-)$. Finally, by (3.23) in \cite{Kel05} we have an isomorphism
\begin{equation}\label{Associativity}
N\otimes_{\cC_2}(M\otimes_{\cC_1} F)\cong (N\otimes_{\cC_2}M)\otimes_{\cC_1} F
\end{equation}
natural in $F\in \cB^{\cC_1}$, $N\in \md \cC_2$ and $M$.

\subsection{Properties of Hom-finite locally bounded categories}\label{Basic properties of Hom-finite locally bounded categories}

Here we use the same terminology as in \cite{DSS17}.
\begin{Definition}\label{Definition:12,5}
Let $\cC$ be a small $k$-linear category.
\begin{enumerate}
\item  $\cC$ is \emphbf{locally bounded} if for any object $c\in \cC$ there are only finitely many objects in $\cC$ mapping nontrivially in and out of $c$. This means that for each $c\in \cC$ we have
\[
\cC(c,c')\neq 0 \quad \text{for only finitely many } c'\in \cC
\]
and
\[
\cC(c'',c)\neq 0 \quad \text{for only finitely many } c''\in \cC;
\]
\item  $\cC$ is \emphbf{Hom-finite} if $\cC(c,c')\in \proj k$ for all $c,c'\in \cC$.
\end{enumerate}
\end{Definition}

\begin{Lemma}\label{Lemma:N5}
Let $\cC$ be a small, $k$-linear, locally bounded, and Hom-finite category. Assume that $M\in \Md\text{-}\cC$ satisfy 
\begin{align*}
& M(c)\in \proj k \text{ }\forall c\in \cC \\
& M(c)\neq 0 \text{ for only finitely many }c\in \cC.
\end{align*}
Then there exists an exact sequence 
\[
\cdots \to Q_2\to Q_1\to Q_0\to M\to 0
\]
where $Q_i$ is a finitely generated projective right $\cC$-module for all $i$. 
\end{Lemma} 

\begin{proof}
Choose an epimorphism $p^c\colon k^{n_c}\to M(c)\to 0$ for each $c\in \cC$ with $M(c)\neq 0$, where $n_c\in \N$. Via the adjunction in \eqref{Adjunction} with $\cC_1=k$, $\cC=\cC_2\op$ and $\cB = \md k$ this corresponds to a morphism $\cC(-,c)\otimes_k k^{n_c}\xrightarrow{p_c} M$ in $\Md\text{-} \cC$. The induced map
\[
\bigoplus _{c\in \cC, \text{ } M(c)\neq 0}\cC(-,c)\otimes_k k^{n_{c}} \xrightarrow{\oplus p_c} M
\]
is then an epimorphism. Let $K$ be the kernel of this map. Then $K(c')\neq 0$ for only finitely many $c'\in \cC$ since the same holds for $\bigoplus _{c\in \cC, \text{ } M(c)\neq 0}\cC(-,c)\otimes_k k^{n_{c}}$. Also, $K(c')$ is the kernel of the epimorphism
\[
\bigoplus _{c\in \cC, \text{ } M(c)\neq 0}\cC(c',c)\otimes_k k^{n_{c}} \xrightarrow{\oplus q_c} M(c')
\]
and since $M(c')\in \proj k$ and $\bigoplus _{c\in \cC, \text{ } M(c)\neq 0}\cC(c',c)\otimes_k k^{n_{c}}\in \proj k$, we get that $K(c')\in \proj k$. Hence, $K$ satisfies the same properties as $M$. We can therefore repeat this construction, which proves the claim.
\end{proof}

In the following we set $D:=\Hom_k(-,k)$. For $B\in \cB$ the functors 
\[
D(-)\otimes_k B\colon (\proj k)\op \to \cB \quad \text{and} \quad \Hom_k(-,B)\colon (\proj k)\op \to \cB
\]
both send $k$ to $B$. Hence, we get an isomorphism
\begin{equation}\label{Dual}
D(V)\otimes_k B \cong \Hom_k(V,B)
\end{equation}
in $\cB$ when $V$ is a finitely generated projective $k$-module. 

\begin{Lemma}\label{Lemma:22}
Assume $\cC$ is a small, $k$-linear, locally bounded and Hom-finite category. Choose an object $c\in \cC$. The following holds:
\begin{enumerate}
	\item\label{Lemma:22,1} $D(\cC(c,-))$ is a finitely presented right $\cC$-module;
	\item\label{Lemma:22,2} We have an isomorphism
	\[
	D(\cC(-,c))\otimes_k B \cong \Hom_k(\cC(-,c),B)
	\]
	in $\cB^{\cC}$ for all $B\in \cB$;
	\item\label{Lemma:22,3} We have an isomorphism
	\[
	\cC(c,-)\otimes_k B \cong \Hom_k(D(\cC(c,-)),B)
	\]
	in $\cB^{\cC}$ for all $B\in \cB$.
\end{enumerate}
\end{Lemma}
\begin{proof}
Statement $\ref{Lemma:22,1}$ follows from Lemma \ref{Lemma:N5}. Statement $\ref{Lemma:22,2}$ and $\ref{Lemma:22,3}$ follow from the isomorphism in \eqref{Dual}.
\end{proof}

\subsection{Nakayama functor for \texorpdfstring{$i_!\dashv i^*$}{} }\label{Comonad with Nakayama functor on}

In this subsection we fix a small, $k$-linear, locally bounded, and Hom-finite category $\cC$. Let $k(\ob\text{-} \cC)$ be the category with the same objects as $\cC$, and with morphisms
\[
k(\ob \text{-}\cC)(c_1,c_2) =
\begin{cases}
0 & \text{if $c_1\neq c_2$},\\
k & \text{if $c_1=c_2$}.
\end{cases}
\]
The functor category $\cB^{k(\ob\text{-} \cC)}$ is just a product of copies of $\cB$, indexed over the objects of $\cC$. Let $i\colon k(\ob\text{-} \cC)\to \cC$ be the inclusion. We have functors 
\begin{align*}
& i_!\colon \cB^{k(\ob\text{-} \cC)}\to \cB^{\cC} \quad i_!((B^c)_{c\in \cC})= \bigoplus_{c\in \cC}\cC(c,-)\otimes_k B^c \\
& i^*\colon \cB^{\cC}\to \cB^{k(\ob\text{-} \cC)} \quad \quad i^*(F)= (F(c))_{c\in \cC} \\
& i_*\colon \cB^{k(\ob\text{-} \cC)}\to \cB^{\cC} \quad i_*((B^c)_{c\in \cC})= \prod_{c\in \cC}\Hom_k(\cC(-,c), B^c). 
\end{align*}

Note that the functors $i_!$ and $i_*$ are well defined since $\cC$ is locally bounded. In fact, evaluating $i_!((B^c)_{c\in \cC})$ and $i_*((B^c)_{c\in \cC})$ on an object in $\cC$ gives a finite sum, and since limits are taken pointwise in $\cB^{\cC}$, it follows that
\begin{equation}\label{Equation:3}
i_!((B^c)_{c\in \cC})= \bigoplus_{c\in \cC}\cC(c,-)\otimes_k B^c= \prod_{c\in \cC}\cC(c,-)\otimes_k B^c
\end{equation}
and
\begin{equation}\label{Equation:4}
i_*((B^c)_{c\in \cC}) = \prod_{c\in \cC}\Hom_k(\cC(c,-), B^c)= \bigoplus_{c\in \cC}\Hom_k(\cC(-,c), B^c).
\end{equation}
Also, $\Hom_{\cC}(\cC(c,-),F)=F(c)=\cC(-,c)\otimes_{\cC}F$, and hence by \eqref{Adjunction} we get that $i^*$ is right adjoint to $i_!$ and left adjoint to $i_*$. Finally, we have functors
\begin{align*}
& \nu\colon \cB^{\cC}\to \cB^{\cC} \quad \nu(F)= D(\cC)\otimes_{\cC}F \\
& \nu^-\colon \cB^{\cC}\to \cB^{\cC} \quad \nu^-(F)= \Hom_{\cC}(D(\cC),F)
\end{align*}
where 
\begin{align*}
& (D(\cC)\otimes_{\cC}F)(c)= D(\cC(c,-))\otimes_{\cC}F \\
& \Hom_{\cC}(D(\cC),F)(c)=\Hom_{\cC}(D(\cC(-,c)),F).
\end{align*}
It follows from \eqref{Adjunction} that $\nu$ is left adjoint to $\nu^-$. 

\begin{Theorem}\label{Theorem:5}
The functor $\nu\colon \cB^{\cC}\to \cB^{\cC}$ is a Nakayama functor relative to $i_!\dashv i^*$
\end{Theorem}

\begin{proof}
For $(B^c)_{c\in \cC}\in \cB^{k(\ob\text{-}\cC)}$ we have natural isomorphisms
\begin{align*}
\nu \circ i_!((B^c)_{c\in \cC})= & \nu (\bigoplus_{c\in \cC}\cC(c,-)\otimes_k B^c) \cong \bigoplus_{c\in \cC}D(\cC(-,c))\otimes_k B^c \\
& \cong \bigoplus_{c\in \cC}\Hom_k(\cC(-,c), B^c) \cong i_*((B^c)_{c\in \cC})
\end{align*}
where the second last isomorphisms follows from Lemma \ref{Lemma:22} part $\ref{Lemma:22,2}$. Hence, $\nu \circ i_!\cong i_*$, and $\nu \circ i_!$ is therefore right adjoint to $i^*$. Also, we have natural isomorphisms 
\begin{multline*}
\nu^-\circ \nu\circ i_! ((B^c)_{c\in \cC}) = \Hom_{\cC}(D(\cC),\bigoplus_{c\in \cC}\Hom_k(\cC(-,c),B^c)) \\ 
\cong \bigoplus_{c\in \cC}\Hom_{\cC}(D(\cC),\Hom_k(\cC(-,c),B^c)) \\
\cong \bigoplus_{c\in \cC}\Hom_k(D(\cC(c,-)),B^c) \cong \bigoplus_{c\in \cC} \cC(c,-)\otimes_k B^c = i_!((B^c)_{c\in \cC}).   
\end{multline*}
This gives an inverse to the unit map $\unit{\nu}{\nu^-}_{i_!}\colon i_!\to \nu^-\circ \nu\circ i_!$. Hence, the claim follows.

\end{proof} 

See Example \ref{Complexes} and Example \ref{Morphism category} for concrete choices of $\cC$. In fact there are interesting examples even when $\cC$ has only one object. In particular, if $\cC=\Lambda$ is a finite-dimensional algebra, then the theorem above applies to the module categories
\[
\Lambda\text{-}\md=(k\text{-}\md)^{\Lambda} \quad \text{and} \quad \Lambda\text{-}\Md=(k\text{-}\Md)^{\Lambda}
\]
see Example \ref{One algebra}. More generally, if $k$ is a commutative ring and $\Lambda$ is a $k$-algebra which is finitely generated projective as a $k$-module, then Theorem \ref{Theorem:5} holds for $\Lambda\text{-}\Md=(k\text{-}\Md)^{\Lambda}$, see Example \ref{One algebra}. If in addition $\Lambda'$ is a $k$-algebra, then the theorem also applies to 
\[
\Lambda'\otimes_k \Lambda\text{-}\Md=(\Lambda'\text{-}\Md)^{\Lambda} 
\]
see Example \ref{Two algebras}. If we furthermore assume that $k$ is coherent or $\Lambda'$ is left coherent then it also holds for the categories of finitely presented modules
\[
\Lambda\text{-}\md=(k\text{-}\md)^{\Lambda} \quad \text{or} \quad \Lambda'\otimes_k \Lambda\text{-}\md=(\Lambda'\text{-}\md)^{\Lambda}
\] 
 see Example \ref{One algebra} and Example \ref{Two algebras}, respectively. 

We now apply Theorem \ref{thm:N2} to Example \ref{Functor category example}.

\begin{Theorem}\label{thm:N3}
Let $k$ be a commutative ring, and let $\cC$ be a small, $k$-linear, locally bounded, and Hom-finite category. Assume 
\[
  \sup_{c\in \cC}(\pdim D(\cC(-,c)))<\infty \quad \text{and} \quad \sup_{c\in \cC}(\pdim D(\cC(c,-)))<\infty
\] 
as left and right $\cC$-modules respectively. Then
\[
\sup_{c\in \cC}(\pdim D(\cC(-,c))) = \sup_{c\in \cC}(\pdim D(\cC(c,-))).
\]
and this number is equal to the Gorenstein dimension of the functor $P=i_!\circ i^*$ on $\cC\text{-}\Md$.
\end{Theorem}

\begin{proof}
By assumption the functor $P=i_!\circ i^*$ on $\cC\text{-}\Md=(k\text{-}\Md)^{\cC}$ is Iwanaga-Gorenstein. Hence, by Theorem \ref{thm:N2} we have that $P$ is $n$-Gorenstein for some $n\in \N$. It follows that
\begin{align*}
& \Ext^n_{\cC\text{-}\Md}(D\cC,-)\neq 0 \quad \text{and} \quad \Ext^i_{\cC\text{-}\Md}(D\cC,-)= 0 \text{ for } i>n \\
& \Tor_n^{\cC}(D\cC,-)\neq 0 \quad \text{and} \quad \Tor_i^{\cC}(D\cC,-)= 0 \text{ for } i>n.
\end{align*}
We therefore get that
\[
\sup_{c\in \cC}(\pdim D(\cC(-,c)))= n =\sup_{c\in \cC}(\flatdim D(\cC(c,-))).
\]
On the other hand, by Lemma \ref{Lemma:N5} there exists an $n$th syzygy of $D(\cC(c,-))$ which is finitely presented. Since finitely presented flat modules are projective, it follows that $\flatdim D(\cC(c,-))= \pdim D(\cC(c,-))$. This proves the claim.
\end{proof}

\begin{Remark}\label{remark:N1}
Locally bounded Hom-finite categories are one of the main object of study \cite{DSS17}. In \cite[Theorem 4.6]{DSS17} they assume that $\cC$ has a Serre functor relative to $k$. This implies that the functor $P$ on $\cC\text{-}\Md$ is $0$-Gorenstein.
\end{Remark}

\bibliography{Mybibtex}

\begin{thebibliography}{10}

\bibitem{AB69}
Maurice Auslander and Mark Bridger.
\newblock {\em Stable module theory}.
\newblock Memoirs of the American Mathematical Society, No. 94. American
  Mathematical Society, Providence, R.I., 1969.

\bibitem{AB89}
Maurice Auslander and Ragnar-Olaf Buchweitz.
\newblock The homological theory of maximal {C}ohen-{M}acaulay approximations.
\newblock {\em M\'em. Soc. Math. France (N.S.)}, (38):5--37, 1989.
\newblock Colloque en l'honneur de Pierre Samuel (Orsay, 1987).

\bibitem{AR91a}
Maurice Auslander and Idun Reiten.
\newblock Cohen-{M}acaulay and {G}orenstein {A}rtin algebras.
\newblock In {\em Representation theory of finite groups and finite-dimensional
  algebras ({B}ielefeld, 1991)}, volume~95 of {\em Progr. Math.}, pages
  221--245. Birkh\"auser, Basel, 1991.

\bibitem{BB69}
Michael Barr and Jon Beck.
\newblock Homology and standard constructions.
\newblock In {\em Sem. on {T}riples and {C}ategorical {H}omology {T}heory
  ({ETH}, {Z}\"urich, 1966/67)}, pages 245--335. Springer, Berlin, 1969.

\bibitem{BK08}
Apostolos Beligiannis and Henning Krause.
\newblock Thick subcategories and virtually {G}orenstein algebras.
\newblock {\em Illinois J. Math.}, 52(2):551--562, 2008.

\bibitem{Bor94a}
Francis Borceux.
\newblock {\em Handbook of categorical algebra. 2}, volume~51 of {\em
  Encyclopedia of Mathematics and its Applications}.
\newblock Cambridge University Press, Cambridge, 1994.
\newblock Categories and structures.

\bibitem{Buc86}
Ragnar-Olaf Buchweitz.
\newblock Maximal {C}ohen-{M}acauley modules and {T}ate-cohomology over
  {G}orenstein rings.
\newblock Unpublished manuscript, 1986.

\bibitem{Che10}
Xiao-Wu Chen.
\newblock Gorenstein homological algebra of artin algebras.
\newblock Postdoctoral Report, USTC, 2010.

\bibitem{CVST10}
J.~Climent~Vidal and J.~Soliveres~Tur.
\newblock Kleisli and {E}ilenberg-{M}oore constructions as parts of biadjoint
  situations.
\newblock {\em Extracta Math.}, 25(1):1--61, 2010.

\bibitem{DSS17}
Ivo Dell'Ambrogio, Greg Stevenson, and Jan Stovicek.
\newblock Gorenstein homological algebra and universal coefficient theorems.
\newblock {\em Math. Z.}, 287(3-4):1109--1155, 2017.

\bibitem{EM65}
Samuel Eilenberg and John~C. Moore.
\newblock Adjoint functors and triples.
\newblock {\em Illinois J. Math.}, 9:381--398, 1965.

\bibitem{EEG08}
E.~Enochs, S.~Estrada, and J.~R. Garcia-Rozas.
\newblock Gorenstein categories and {T}ate cohomology on projective schemes.
\newblock {\em Math. Nachr.}, 281(4):525--540, 2008.

\bibitem{EJ95}
Edgar~E. Enochs and Overtoun M.~G. Jenda.
\newblock Gorenstein injective and projective modules.
\newblock {\em Math. Z.}, 220(4):611--633, 1995.

\bibitem{EJ11}
Edgar~E. Enochs and Overtoun M.~G. Jenda.
\newblock {\em Relative homological algebra. {V}olume 1}, volume~30 of {\em De
  Gruyter Expositions in Mathematics}.
\newblock Walter de Gruyter GmbH \& Co. KG, Berlin, extended edition, 2011.

\bibitem{EJ11a}
Edgar~E. Enochs and Overtoun M.~G. Jenda.
\newblock {\em Relative homological algebra. {V}olume 2}, volume~54 of {\em De
  Gruyter Expositions in Mathematics}.
\newblock Walter de Gruyter GmbH \& Co. KG, Berlin, 2011.

\bibitem{GM03}
Sergei~I. Gelfand and Yuri~I. Manin.
\newblock {\em Methods of homological algebra}.
\newblock Springer Monographs in Mathematics. Springer-Verlag, Berlin, second
  edition, 2003.

\bibitem{Hol04}
Henrik Holm.
\newblock Gorenstein homological dimensions.
\newblock {\em J. Pure Appl. Algebra}, 189(1-3):167--193, 2004.

\bibitem{Iya01}
Osamu Iyama.
\newblock Representation theory of orders.
\newblock In {\em Algebra---representation theory ({C}onstanta, 2000)},
  volume~28 of {\em NATO Sci. Ser. II Math. Phys. Chem.}, pages 63--96. Kluwer
  Acad. Publ., Dordrecht, 2001.

\bibitem{J07}
Peter J{\o}rgensen.
\newblock Existence of {G}orenstein projective resolutions and {T}ate
  cohomology.
\newblock {\em J. Eur. Math. Soc. (JEMS)}, 9(1):59--76, 2007.

\bibitem{Kel05}
G.~M. Kelly.
\newblock Basic concepts of enriched category theory.
\newblock {\em Repr. Theory Appl. Categ.}, (10):vi+137, 2005.
\newblock Reprint of the 1982 original [Cambridge Univ. Press, Cambridge;
  MR0651714].

\bibitem{Lau06}
Aaron~D. Lauda.
\newblock Frobenius algebras and ambidextrous adjunctions.
\newblock {\em Theory Appl. Categ.}, 16:No. 4, 84--122, 2006.

\bibitem{MLan98}
Saunders Mac~Lane.
\newblock {\em Categories for the working mathematician}, volume~5 of {\em
  Graduate Texts in Mathematics}.
\newblock Springer-Verlag, New York, second edition, 1998.

\bibitem{She16}
Dawei Shen.
\newblock A description of {G}orenstein projective modules over the tensor
  products of algebras, 2016.
\newblock arXiv:1602.00116.

\bibitem{Sto14}
Jan Stovicek.
\newblock Derived equivalences induced by big cotilting modules.
\newblock {\em Adv. Math.}, 263:45--87, 2014.

\bibitem{Wei94}
Charles~A. Weibel.
\newblock {\em An introduction to homological algebra}, volume~38 of {\em
  Cambridge Studies in Advanced Mathematics}.
\newblock Cambridge University Press, Cambridge, 1994.

\bibitem{Zak69}
Abraham Zaks.
\newblock Injective dimension of semi-primary rings.
\newblock {\em J. Algebra}, 13:73--86, 1969.

\end{thebibliography}
\bibliographystyle{plain} 

\end{document}